%% file: Steklov-submit_v2.tex
\documentclass[hidelinks,onefignum,onetabnum]{siamart250211}

\usepackage{multirow}
\usepackage{subcaption}
\usepackage[normalem]{ulem}

\input{ex_shared_v1}

\ifpdf
\hypersetup{
	pdftitle={DEN for Parametric Non-self-adjoint Eigenvalue Problems},
	pdfauthor={H. Li, J. Sun, and Z. Zhang}
}
\fi

\begin{document}
	
	\maketitle
	
	\begin{abstract}
        We consider operator learning for efficiently solving parametric non-self-adjoint eigenvalue problems. To overcome the spectral instability and mode switching associated with non-self-adjoint operators, we choose to learn the eigenspace rather than individual eigenfunctions. In particular, we propose a Deep Eigenspace Network (DEN) architecture integrating Fourier Neural Operators, geometry-adaptive POD bases, and explicit banded cross-mode mixing mechanism to capture complex spectral dependencies. We apply DEN to the non-self-adjoint Steklov eigenvalue problem and prove the Lipschitz continuity of the eigenspace with respect to the parameter. Furthermore, we derive error bounds for the eigenvalues. Numerical experiments validate that DEN is highly effective and efficient.
	\end{abstract}
	
	


    \section{Introduction}

    Deep learning approaches have gained increasing attention for solving eigenvalue problems of partial differential equations. Early efforts primarily focused on the construction of neural networks to solve a single, often high-dimensional problem. Approaches based on the Deep Ritz method, semigroup approaches, Tensor Neural Networks, and neural network-based supspace method have been developed to approximate eigenpair(s) \cite{han2020solving, li2022semigroup, ji2024deep, wang2024computing, ben2023deep, dai2024subspace}. 
    In contrast, operator learning enables the prediction of spectral properties for a class of parameters \cite{li2025operator}. Once trained, the neural operator provides real-time inference, a desired property in many-query tasks such as inverse problems. We consider operator learning of the parameter-to-solution mapping for non-self-adjoint eigenvalue problems. Compared to the self-adjoint case, the non-self-adjoint nature leads to a fundamental challenge: direct regression of individual eigenfunctions is numerically intractable due to the inherent spectral instability. According to matrix perturbation theory \cite{stewart1990matrix}, spectral clustering and mode crossings cause individual eigenvectors to rotate rapidly or switch indices, rendering the mapping from parameters to eigenfunctions highly non-smooth or discontinuous. To overcome this challenge, we pivot to learning the eigenspace, which remains stable even when individual eigenfunctions within the cluster fluctuate. 

    In this paper, we propose a \textbf{Deep Eigenspace Network (DEN)} to predict eigenspaces using a geometry-adaptive basis that naturally handles unstructured data, followed by the derivation of a Proper Orthogonal Decomposition (POD) basis from the output snapshots. 
    To model the spectral couplings, we propose a spectral convolution layer with explicit cross-mode connections, parameterized via low-rank banded matrices. This enables DEN to capture intricate off-diagonal dependence and maintain the computational efficiency. Such inter-modal dependence extends beyond the diagonal structure of standard spectral layers. Furthermore, DEN accommodates the unstructured discretizations by the Finite Element Methods in the training data while retaining global receptive field advantage over local message-passing Graph Neural Networks (GNNs) \cite{corso2024graph}. An essential feature of DEN is its generalization of spectral transform beyond the uniform-grid restriction.

    DEN belongs into the burgeoning operator learning paradigms \cite{lu2019deeponet, li2020neural, li2020multipole, liu2024render, wang2022deepparticle, cheng2025podno, loeffler2024graph}. There exist few works on operator learning for spectral problems. Li et al. \cite{li2025operator} utilized Fourier Neural Operators (FNO) and CNNs to predict the eigenvalues and eigenfunctions of the Laplacian on variable domains. Chang et al. \cite{chang2024neural} employed PINNs to track the continuous evolution of eigenfunctions with respect to shape deformation. However, these works either address only the self-adjoint problems or rely on physics-informed loss functions that face challenges in the non-self-adjoint case. 
    


    We implement DEN for the parametric non-self-adjoint Steklov eigenvalue problem  \cite{kuznetsov2014legacy, canavati1979discontinuous, cao2013multiscale, kuznetsov2014legacy}.  Various discretization techniques have been proposed for the Steklov eigenvalue problem, including isoparametric finite element methods (FEM) \cite{andreev2004isoparametric}, virtual element methods \cite{mora2015virtual}, non-conforming FEM \cite{russo2011posteriori}, spectral-Galerkin methods \cite{an2016highly}, and advanced adaptive or multilevel schemes \cite{bi2016adaptive, han2015multilevel}. The majority of these contributions consider the self-adjoint case. The non-self-adjoint Steklov eigenvalue problem, which arises in the modeling of inhomogeneous absorbing media \cite{cakoni2016stekloff}, remains under-explored numerically.  Bramble and Osborn \cite{bramble1972approximation} investigated the problem assuming the uniform ellipticity, but did not provide numerical realizations.  Discretization of non-self-adjoint operators leads to non-Hermitian generalized eigenvalue problems (GEPs) \cite{liu2019spectral}, which are difficult to solve \cite{saad2011numerical}. 
    

    The remainder of this paper is organized as follows. In Section \ref{sec:DEN}, we introduce the architecture of DEN. Section \ref{sec:Method} presents the specific subspace learning framework for the parametric Steklov eigenvalue problem, including the loss formulation on the Grassmann manifold and the Rayleigh-Ritz procedure. In Section \ref{sec:Theory}, we provide a rigorous analysis regarding the stability of the eigenspace with respect to parameter perturbations and derive error bounds for the predicted eigenvalues. Section \ref{sec:Exp} reports numerical experiments to demonstrate the accuracy, efficiency, and robustness of DEN. Finally, Section \ref{sec:Concl} contains conclusions and future work.

    \section{Deep Eigenspace Network}\label{sec:DEN}
    In this section, we present the structure of DEN (see Fig.~\ref{fig:DEN}), discussing the cross-mode mixing operator, basis selection, and banded low-rank parametrization.
    

    \begin{figure}[htbp]
        \centering
        \includegraphics[width=1\textwidth]{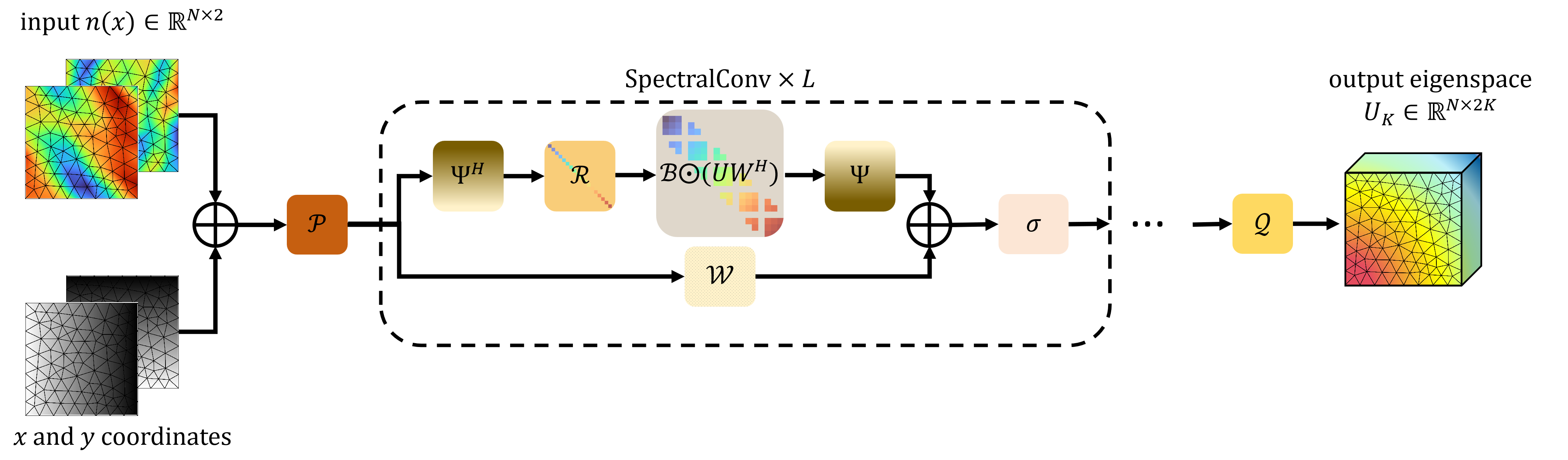}
        \caption{Structure of Deep Eigenspace Network.}
        \label{fig:DEN}
    \end{figure}

    \subsection{Structure of DEN}

    Let $\mathbf{n} \in \mathbb{R}^{N \times d_{\text{in}}}$ denote the discretized input function defined on an unstructured mesh with $N$ nodes (e.g., the variable coefficient $n(x)$), and let $\mathbf{y} \in \mathbb{R}^{N \times d_{\text{out}}}$ denote the target output (e.g., the eigenspace). Building on the FNO framework, DEN consists of three components:
    
    \begin{enumerate}
        \item \textbf{Lifting Layer} $\mathcal{P}$: A pointwise linear transformation that maps the input into a higher-dimensional latent representation:
        \begin{equation}
            \mathbf{z}^{(0)} = \mathcal{P}(\mathbf{n}) \in \mathbb{R}^{N \times d_{0}}.
        \end{equation}
    
        \item \textbf{Iterative Spectral Convolution Layers}: The latent representation is processed through a sequence of $L$ spectral convolution blocks. For $l = 0, \dots, L-1$:
        \begin{equation}\label{eq:DEN_layer}
            \mathbf{z}^{(l+1)} = \sigma \left(
            \Psi \, M_{BLR} \left( R \cdot \Psi^{H} \mathbf{z}^{(l)} \right)
            + \mathcal{W} \mathbf{z}^{(l)} + \mathbf{b}^{(l)}
            \right),
        \end{equation}
        where $\Psi \in \mathbb{R}^{N \times K}$ is a geometry-adaptive orthogonal basis, $\mathcal{W}$ is a pointwise linear operator in the physical space, $\mathbf{b}^{(l)}$ is a learnable bias, and $\sigma$ is a nonlinear activation function (e.g., GeLU). The two spectral operators $R$ and $M_{BLR}$ act sequentially in the generalized coefficient space.
        \begin{itemize}
            \item The \textbf{channel mixing operator} $R$ mixes information across input channels independently for each mode. Denoting $\hat{\mathbf{z}} =
            \Psi^H \mathbf{z}^{(l)} \in \mathbb{C}^{K \times d_l}$, its output is:
            \begin{equation}\label{eq:channel_mix}
                (R \cdot \hat{\mathbf{z}})_{c_{\text{out}}, m} =
                \sum_{c_{\text{in}}=1}^{d_l} R_{c_{\text{in}}, c_{\text{out}}, m}
                \, \hat{\mathbf{z}}_{c_{\text{in}}, m},
            \end{equation}
            where $R \in \mathbb{C}^{d_l \times d_{l+1} \times K}$ is a learnable complex-valued tensor, and mixing is performed per mode without cross-mode interaction.
            \item The \textbf{cross-mode mixing operator} $M_{BLR}$ then couples different spectral modes independently for each output channel:
            \begin{equation}\label{eq:mode_mix}
                (M_{BLR} \, \hat{\mathbf{z}})_{c_{\text{out}}} =
                M_{BLR} \cdot \hat{\mathbf{z}}_{c_{\text{out}}},
            \end{equation}
            where $M_{BLR} \in \mathbb{C}^{K \times K}$ is a shared banded low-rank matrix applied uniformly across channels (detailed in Sections~\ref{sec:ModeMix} and \ref{sec:BandLowRank}).
        \end{itemize}
        The summation of the spectral branch and the local residual branch ($\mathcal{W}$) enables each block to integrate global spectral interactions with local spatial features.
    
        \item \textbf{Projection Layer} $\mathcal{Q}$: A pointwise linear map that projects
        the latent features back to the target dimension, yielding the output
        $\mathbf{y} = \mathcal{Q}(\mathbf{z}^{(L)}) \in \mathbb{R}^{N \times d_{\text{out}}}$.
    \end{enumerate}
    
    The spectral convolution layer of DEN incorporates two key innovations.
    
    \paragraph{a. Cross-Mode Mixing for Spectral Couplings}
    Standard FNO treats each Fourier mode independently in the spectral convolution layers (cf.\ Eq.~\eqref{eq:channel_mix}), which is effective for translation-invariant operators where modes are approximately decoupled. Non-self-adjoint operators, however, are characterized by strong interactions between different eigenmodes. DEN explicitly models these off-diagonal dependencies by introducing the cross-mode mixing operator $M_{BLR}$ (cf.\ Eq.~\eqref{eq:mode_mix}), parameterized as a banded low-rank matrix to balance representational capacity with computational efficiency.
    
    \paragraph{b. Geometry-Adaptive Spectral Basis}
    Standard FNO is built on the discrete Fourier transform, which is tailored to structured Cartesian grids. To accommodate unstructured meshes arising from finite element discretizations \cite{SunZhou2016}, DEN replaces the fixed sinusoidal basis with a geometry-adaptive orthogonal basis $\Psi$, following the Generalized Spectral Operator framework \cite{cheng2025podno, loeffler2024graph}. This retains the global receptive field of spectral methods while naturally handling arbitrary geometries. As shown in Section~\ref{sec:basis}, the Proper Orthogonal Decomposition (POD) basis derived from the snapshot matrix of the target outputs constitutes the optimal choice for $\Psi$, effectively aligning the network's latent representation with the intrinsic manifold of the eigenfunctions.

    \subsection{Cross-Mode Mixing}\label{sec:ModeMix}
    Let \(\Psi\) be an orthogonal transformation basis and \(A\) be the target linear operator acting on the physical domain. The spectral representation is given by a transformation \(A_{\Psi} := \Psi^{-1} A \Psi\). For PDEs involving variable coefficients, complex boundary conditions, or non-translation-invariant phenomena, \(A_{\Psi}\) is generally \emph{non-diagonal}. Energy spreads across frequencies, inducing non-trivial cross-mode coupling.

    To address the limitation of diagonal approximations, we introduce a \emph{cross-mode mixing operator}. Let \(M \in \mathbb{C}^{K \times K}\) be a matrix that couples different modes.
    The enhanced spectral convolution becomes:
    \begin{equation}
    \mathbf{y} = \Psi\, M\, (R\cdot \Psi^{H} \mathbf{x}).
    \end{equation}
    This formulation expands the hypothesis space from diagonal operators to the full space of linear operators on the span of \(\Psi\). The following proposition claims that, for a finite-dimensional non-self-adjoint operator, adding cross-mode interaction is necessary to eliminate representation error.
    
    \begin{proposition}[Irreducible Error of Diagonal Approximations]
    \label{prop:spectral_capacity}
    Let \(A_{\Psi} \in \mathbb{C}^{K \times K}\) be the spectral representation of a non-Hermitian $A$. Let \(\mathcal{D} \subset \mathbb{C}^{K \times K}\) be the subspace of diagonal matrices, and let \(\mathcal{M} = \mathbb{C}^{K \times K}\) be the space of arbitrary matrices.
    Define the approximation errors for the diagonal and full models as:
    \begin{equation}
    \varepsilon_{\mathcal{D}} = \min_{D \in \mathcal{D}} \| A_{\Psi} - D \|_F, \quad \text{and} \quad \varepsilon_{\mathcal{M}} = \min_{M \in \mathcal{M}} \| A_{\Psi} - M \|_F,
    \end{equation}
    where \(\| \cdot \|_F\) denotes Frobenius norm. It holds that
    \begin{equation}
    \varepsilon_{\mathcal{D}} > 0 = \varepsilon_{\mathcal{M}}.
    \end{equation}
    \end{proposition}
    The proof is straightforward and thus is omitted. It shows that the diagonal model suffers from an irreducible error equal to the energy of the off-diagonal terms.


    \subsection{Basis Selection}\label{sec:basis}


    With a cross-mode mixing layer, we discuss the impact of basis selection on model performance. While prior works \cite{cheng2025podno, loeffler2024graph} have employed various basis choices for different problems, we find experimentally and theoretically that the output-aligned POD basis, $\text{POD}(Y)$, yields the best performance.

    \begin{definition}[\textbf{POD}$(Y)$]
    Let the snapshot matrix of the output eigenspaces be $Y = [V_{n,1}, \dots, V_{n,N_{\text{train}}}]
    \in \mathbb{C}^{N \times n N_{\text{train}}}$, where each $V_{n,i} \in \mathbb{C}^{N \times n}$ is the eigenspace (of dimension $n$) corresponding to the $i$-th training sample, and $N_{\text{train}}$ is the number of training samples. The basis $\text{POD}(Y)$ is defined as the span of the first $K$ left singular vectors of $Y$ (equivalently, the dominant eigenvectors of $YY^H$).
    \end{definition}

    

    To understand why $\text{POD}(Y)$ is preferable, we introduce the following framework. Any basis $\Psi$ of dimension $K$ induces a decomposition of the spectral feature space $\mathcal{Z} \cong \mathbb{R}^K$ into two orthogonal components:
    \begin{itemize}
        \item $\mathcal{S}_Y \subset \mathcal{Z}$: the \textbf{target subspace} (signal manifold) with $\dim(\mathcal{S}_Y) = d_Y$, spanned by basis directions aligned with the output, and
        \item $\mathcal{S}_\perp \subset \mathcal{Z}$: the \textbf{noise subspace}
        with $\dim(\mathcal{S}_\perp) = d_\perp$, such that $\mathcal{Z} = \mathcal{S}_Y \oplus \mathcal{S}_\perp$.
    \end{itemize}
    A basis is called \textit{entangled} if its vectors mix signal and noise components, resulting in non-negligible $\|z_\perp\|$ for typical inputs; it is called \textit{aligned} if $\mathcal{S}_\perp \approx \{0\}$, i.e., $z_\perp \approx 0$ for all $z \in \mathcal{Z}$. By construction, $\text{POD}(Y)$ is an aligned basis.
    
    Any entangled basis incurs a capacity penalty on the cross-mode mixing operator $M \in \mathbb{R}^{K \times K}$, formalized as follows.
    

    \begin{theorem}[\textbf{Effective Rank Contraction}]
    Let the learning objective be to approximate the dynamics $D$ on $\mathcal{S}_Y$ while suppressing the noise from $\mathcal{S}_\perp$. That is, the learned operator $M$ must satisfy:
    \begin{enumerate}
        \item \textbf{Denoising Condition:} $M z_\perp = 0, \quad \forall z_\perp \in \mathcal{S}_\perp$.
        \item \textbf{Evolution Condition:} $M z_Y \approx D z_Y, \quad \forall z_Y \in \mathcal{S}_Y$.
    \end{enumerate}
    If $\text{rank}(M) \le r_{max}$, where $r_{max}$ is the rank budget for weight matrix $M$, then the effective rank available to approximate the physical dynamics $D$ is bounded by:
    \begin{equation}
        \text{rank}(M|_{\mathcal{S}_Y}) \le r_{max} - d_\perp
    \end{equation}
    \end{theorem}

    \begin{proof}
        The result follows directly from the Rank-Nullity Theorem: the Denoising Condition
        requires $\mathcal{S}_\perp \subseteq \mathcal{N}(M)$, which consumes at least
        $d_\perp$ degrees of freedom from the rank budget $r_{max}$, leaving at most
        $r_{max} - d_\perp$ for approximating $D$ on $\mathcal{S}_Y$.
    \end{proof}

    
    Following the establishment of the effective rank contraction caused by entangled bases, we now identify the optimal basis configuration. To maximize the representational capacity of the spectral layer, the basis should be strictly aligned with the geometry of the target signal manifold.
    
    \begin{theorem}[\textbf{Optimality of the Output-Aligned Basis}]
    \label{thm:optimality_pod_y}
    Let \( Y \in \mathbb{R}^{N \times d_h} \) be the snapshot matrix of the target physical states (the solution manifold), and let \( C_Y = \mathbb{E}[y y^\top] \) be the covariance operator of the target signals.
    Assume the spectral mixing layer seeks a rank-\( K \) approximation of the operator output.
    The projection error energy:
    \begin{equation}
    \mathcal{E}(\Psi) = \mathbb{E}_{y \sim \mathcal{S}_Y} \left[ \| y - \Psi \Psi^H y \|_2^2 \right]
    \end{equation}
    is minimized if and only if the basis \( V \) spans the same subspace as the first \( K \) principal components of \( Y \) (denoted as \(\text{POD}(Y)\)). Consequently, under this basis, the dimension of the irrelevant subspace vanishes (\( d_\perp = 0 \)), and the effective rank is maximized:
    \begin{equation}
    \text{rank}(M|_{\mathcal{S}_Y}) = \text{rank}(M).
    \end{equation}
    \end{theorem}
    
    \begin{proof}
    The proof follows from the variational characterization of the singular value decomposition (Eckart-Young-Mirsky theorem \cite{eckart1936approximation}).
    
    \end{proof}
    
    This result justifies decoupling the input and output bases. While an entangled basis attempts to span both the input condition space and the solution space, the operator learning task is strictly directional ($X \to Y$). By aligning the spectral layer's codomain with $\text{POD}(Y)$, we ensure the matrix $M$ focuses solely on the destination manifold.

    \subsection{Banded Low-Rank Parameterization}\label{sec:BandLowRank}

    While Proposition \ref{prop:spectral_capacity} suggests that a dense mixing matrix \(M\) minimizes representation error, it introduces significant optimization difficulties. In practice, operators arising from physical systems exhibit specific spectral structures, namely, energy concentration and local spectral coupling. Ignoring these structures leads to poor convergence and overfitting. To address this, we propose a \textbf{Banded Low-Rank (BLR)} parameterization.
    
    \paragraph{Low-Rank Property of the Spectral Operator}
    The first structural prior is the low effective rank of the target operator in the POD basis. By construction, the POD basis vectors \(\Psi = \{\psi_k\}_{k=1}^K\) are ordered by their energy contribution to the dataset. Consequently, the output of the physical operator, \(y = \mathcal{A}(x)\), resides primarily in the subspace spanned by the dominant modes.
    
    This implies that the spectral representation \(A_{\Psi} = \Psi^H A \Psi\) is \emph{numerically low-rank}, i.e., its singular values decay rapidly. A full-rank parameterization of \(M\) is therefore over-parameterized, learning noise correlations in the tail of the spectrum rather than physical dynamics. We thus propose a low-rank mixing operator:
    \begin{equation}
    M_{LR} = U W^{H}, \quad \text{where } U, W \in \mathbb{C}^{K \times r}, \, r \ll K.
    \end{equation}
    
    \paragraph{Optimization Stability via Banded Constraints}
    The second, and critical, structural prior concerns the \emph{ordering} of the spectrum. A standard low-rank factorization \(M_{LR} = UW^H\) is globally coupled; at initialization, it connects every frequency mode to every other mode with non-zero probability. This causes two optimization issues.
    \begin{enumerate}
        \item \textbf{Spectral Pollution:} Random initialization allows high-frequency noise to propagate directly into dominant low-frequency modes (from high indices to low indices). This disrupts the coarse-grained physical features early in training.
        \item \textbf{Optimization Landscape:} The model is forced to expend early training iterations ``unlearning'' these unphysical long-range couplings to restore the correct spectral hierarchy, leading to a rugged optimization landscape.
    \end{enumerate}
    
    To stabilize training, we introduce a banded constraint. Physical energy transfer typically occurs between spectrally adjacent modes (e.g., energy cascade in fluids) rather than arbitrary non-local jumps. We enforce this locality by applying a band mask \(\mathcal{B}\) to the low-rank matrix, ensuring that interactions are restricted to a bandwidth \(b\) along the diagonal.
    
    We define the binary mask matrix \(\mathcal{B} \in \{0, 1\}^{K \times K}\) as:
    \begin{equation}
    \mathcal{B}_{ij} = 
    \begin{cases} 
    1 & \text{if } |i - j| \le b, \\
    0 & \text{otherwise},
    \end{cases}
    \end{equation}
    where \(b\) is a hyperparameter controlling the spectral interaction bandwidth. The final cross-mode mixing operator \(M\) is defined as the Hadamard product of the mask and the low-rank component:
    \begin{equation}
    M_{BLR} = \mathcal{B} \odot (U W^{H}),
    \end{equation}
    where $\odot$ denotes the element-wise multiplication, i.e., Hadamard product. The low-rank structure \(UW^H\) captures the dominant subspace of the operator, while the banded mask \(\mathcal{B}\) imposes an inductive bias that respects the spectral ordering, preventing noise injection and accelerating convergence.

    \section{Methodology}\label{sec:Method}
    In this section, we implement DEN for the parametric Steklov eigenvalue problem and analyze the specific challenges introduced by the non-self-adjoint nature of the operator. The network structure and the corresponding loss function are specified. Following the prediction of the eigenspace, the Rayleigh-Ritz procedure is employed to recover the eigenvalues and eigenfunctions. Finally, the data generation process is described.

    \subsection{Problem Setup}

    \subsubsection{The non-self-adjoint Steklov Eigenvalue Problem}
    Let $\Omega \subset \mathbb{R}^2$ be a bounded domain with a Lipschitz boundary $\partial \Omega$. The parametric Steklov eigenvalue problem is to find the eigenpair $(\lambda, u)$ such that:
    \begin{equation} \label{eq:steklov_pde}
        \begin{cases}
        \begin{aligned}
            \Delta u + k^2 n({x}) u &= 0 \quad &&\text{in } \Omega, \\
            \frac{\partial u}{\partial \nu} + \lambda u &= 0 \quad &&\text{on } \partial \Omega,
        \end{aligned}
        \end{cases}
    \end{equation}
    where $\nu$ denotes the outward unit normal vector, $k$ is the wavenumber, and $n({x})=n_1(x) + i\frac{n_2(x)}{k}$ is the refractive index parameter, $n_1(x)>0,\ n_2(x)\geq 0$ are bounded smooth functions.

    The variational formulation corresponding to \eqref{eq:steklov_pde} is to find $\lambda \in \mathbb{C}$ and $u \in H^1(\Omega) \setminus \{0\}$ such that
    \begin{equation}\label{WeakForm}
        a(u, v) = \lambda b(u, v) \quad \forall v \in H^1(\Omega),
    \end{equation}
    where the sesquilinear forms $a(\cdot, \cdot)$ and $b(\cdot, \cdot)$ are defined as:
    \begin{equation}
        a(u, v) = -\int_{\Omega} \nabla u \cdot \nabla \bar{v} \, d{x} + k^2 \int_{\Omega} n({x}) u \bar{v} \, d{x}, \quad
        b(u, v) = \int_{\partial \Omega} u \bar{v} \, ds.
    \end{equation}

    When the parameter $n({x})$ is complex-valued (absorbing medium), the problem is \textbf{non-self-adjoint}. This introduces several theoretical and numerical complexities compared to the self-adjoint case.
    \begin{itemize}
        \item \textbf{Complex Spectrum:} The eigenvalues $\lambda$ are complex-valued ($\lambda \in \mathbb{C}$). Sorting is typically performed based on the modulus $|\lambda|$ or real part, which is susceptible to reordering under perturbation.
        \item \textbf{Non-Orthogonality:} The eigenfunctions $\{u_i\}$ are no longer mutually orthogonal with respect to the standard inner product. 
        \item \textbf{Invalidity of Min-Max Principles:} 
        For self-adjoint eigenvalue problems (where $n({x})$ is real), the Rayleigh quotient 
        \begin{equation}
            \mathcal{R}(u) = \frac{a(u, u)}{b(u, u)}
        \end{equation}
        provides a variational characterization of the spectrum. The eigenvalues can be determined via the Courant-Fischer min-max principle:
        \begin{equation}
            \lambda_k = \min_{V_k \subset H^1, \dim(V_k)=k} \max_{u \in V_k, u \neq 0} \mathcal{R}(u).
        \end{equation}
        This property is instrumental in many physics-informed learning frameworks. For instance, Dai et al. \cite{dai2024subspace} leverage this property to construct loss functions based on minimizing the Rayleigh quotient, thereby training neural networks to find eigenpairs without labeled data. However, in the non-self-adjoint case, the above property does not hold. 
    \end{itemize}

    \subsubsection{Discrete Formulation and Mode Crossing}
    FEM discretization of \eqref{WeakForm} leads to the generalized matrix eigenvalue problem:
    \begin{equation} \label{eq:GEP}
        \mathbf{A}(n) \mathbf{x} = \lambda \mathbf{B} \mathbf{x}.
    \end{equation}
    The matrices are defined as $\mathbf{A}(n) = -\mathbf{S} + k^2 \mathbf{M}_n$ and $\mathbf{B} = \mathbf{M}_{\partial \Omega}$, where $\mathbf{S}$ is the stiffness matrix, $\mathbf{M}_n$ is the domain mass matrix weighted by $n({x})$, and $\mathbf{M}_{\partial \Omega}$ is the boundary mass matrix (see \cite{liu2019spectral}).

    Our goal is to learn the operator $\mathcal{G}: n({x}) \mapsto \{(\lambda_i, u_i)\}_{i=1}^K$, mapping the parameter $n({x})$ to the first $K$ eigenpairs (sorted by the modulus of the eigenvalues). One might try to train neural operators to regress $\lambda_i$ and $u_i$ directly similar to \cite{li2025operator}. However, preliminary investigations reveal a fundamental difficulty: while eigenvalues generally exhibit continuous dependence on parameters, the eigenfunctions are difficult to predict. This stems from the behavior of non-self-adjoint spectra under perturbation. Due to the classic perturbation theory in \cite{stewart1990matrix}, the sensitivity of an eigenvector is inversely proportional to the separation between its associated eigenvalue and the rest of the spectrum. When eigenvalues form a cluster or approach a crossing point as $n(x)$ varies, the individual eigenvectors become ill-conditioned and can rotate rapidly within the invariant subspace. In practical computations, this phenomenon manifests as "index switching" or highly non-smooth behavior in the mapping $n(x) \mapsto u_i$. Consequently, these discontinuities pose significant challenges for mode-matching algorithms and necessitate substantial training data.

    \begin{figure}[!htp]
        \centering
        \includegraphics[width=1\linewidth]{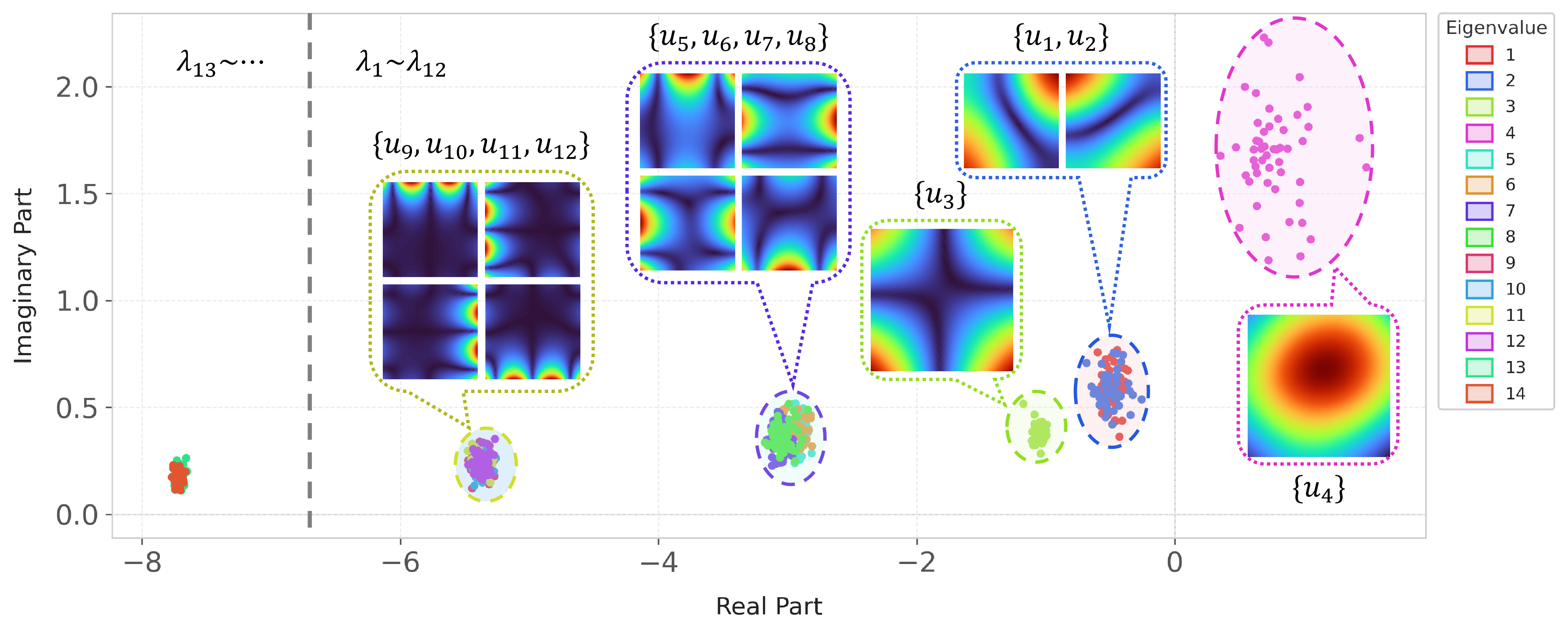}
        \caption{Demonstration of eigenvalue clusters for the first 14 eigenvalues.
        Each dot represents a sampled eigenvalue of a particular refractive index $n(x)$,
        colored by eigenvalue index assigned via eigenfunction pattern matching.
        The dashed boundaries delineate five clusters for the first 12 eigenvalues:
        $\{\lambda_1, \lambda_2\}$, $\{\lambda_3\}$, $\{\lambda_4\}$,
        $\{\lambda_5, \lambda_6, \lambda_7, \lambda_8\}$, and
        $\{\lambda_9, \lambda_{10}, \lambda_{11}, \lambda_{12}\}$.
        The modulus of representative eigenfunctions of each cluster are shown in the insets.
        $\lambda_{13}$ and beyond are separated from the first 12 by the dashed vertical line.}
        \label{fig:ev_cluster}
    \end{figure}

    As illustrated in Fig.~\ref{fig:ev_cluster}, the first 12 eigenvalues naturally organize into five well-separated clusters in the complex plane:
    $\{\lambda_1, \lambda_2\}$, $\{\lambda_3\}$, $\{\lambda_4\}$,
    $\{\lambda_5, \lambda_6, \lambda_7, \lambda_8\}$, and
    $\{\lambda_9, \lambda_{10}, \lambda_{11}, \lambda_{12}\}$.
    Eigenvalues belonging to the same cluster share similar eigenfunction
    modal structures, as shown by the inset visualizations.
    Within each cluster, the eigenvalues are closely interspersed and their
    ordering by modulus frequently changes across different realizations of $n(x)$,
    making magnitude-based indexing unreliable.
    In contrast, eigenvalues belonging to different clusters remain well separated in the complex plane across all samples, enabling unambiguous discrimination between clusters. This cluster structure motivates us to learn the mapping from $n({x})$ to eigenspace $\mathcal{U}_K = \text{span}\{u_1, \dots, u_K\}$. Then the individual eigenpairs are recovered via the Rayleigh-Ritz procedure, which solves a small-scale projected eigenvalue problem (see Section~\ref{sec:RayeighRitz}). We shall prove that the eigenspace associated with a cluster of eigenvalues is stable, provided the cluster is separated from the rest of the spectrum (see Section \ref{sec:SubStable} on subspace stability).


    \subsection{Subspace Learning Framework}

    Recently, Fanaskov et al. \cite{fanaskov2025deep} introduced a subspace regression using deep learning. Let $\mathcal{V} \subset \mathbb{C}^N$ be the target eigenspace spanned by the first $k_V$ eigenfunctions, represented by the matrix $\mathbf{V} \in \mathbb{C}^{N \times k_V}$. Our goal is to train a neural operator to predict a subspace $\mathcal{U}$, represented by $\mathbf{U} \in \mathbb{C}^{N \times k_U}$, such that $\mathcal{V}$ is embedded within $\mathcal{U}$ ($k_V \leq k_U$). The condition $\mathcal{V} \subseteq \mathcal{U}$ ensures that the true eigenvectors can be recovered as linear combinations of the basis vectors of $\mathcal{U}$.

    Ideally, the loss function $\mathcal{L}(\mathbf{U}, \mathbf{V})$ should satisfy two properties.
    \begin{enumerate}
        \item[(i)] \textbf{Basis Invariance:} The loss should depend only on the subspaces spanned by the matrices, not on their specific basis representations. That is, for any invertible matrices $\mathbf{S} \in \mathbb{C}^{k_U \times k_U}$ and $\mathbf{T} \in \mathbb{C}^{k_V \times k_V}$:
        \begin{equation}
            \mathcal{L}(\mathbf{U}\mathbf{S}, \mathbf{V}\mathbf{T}) = \mathcal{L}(\mathbf{U}, \mathbf{V}).
        \end{equation}
        \item[(ii)] \textbf{Inclusion Monotonicity:} The loss should be non-negative, reaching zero if and only if the inclusion condition is met:
        \begin{equation}
            \mathcal{L}(\mathbf{U}, \mathbf{V}) \geq 0, \quad \text{and} \quad \mathcal{L}(\mathbf{U}, \mathbf{V}) = 0 \iff \text{span}(\mathbf{V}) \subseteq \text{span}(\mathbf{U}).
        \end{equation}
        Note that when $k_U > k_V$, the loss is asymmetric, i.e., $\mathcal{L}(\mathbf{U}, \mathbf{V}) \neq \mathcal{L}(\mathbf{V}, \mathbf{U})$.
    \end{enumerate}

    A loss function satisfying these criteria is given in \cite{fanaskov2025deep}, which measures the orthogonality defect between the subspaces:
    \begin{equation} \label{eq:L1_loss}
        \mathcal{L}(\mathbf{U}, \mathbf{V}) = k_V - \|\mathbf{Q}_V^H \mathbf{Q}_U\|_F^2,
    \end{equation}
    where $\mathbf{Q}_U$ and $\mathbf{Q}_V$ are the orthonormal bases obtained from $\mathbf{U}$ and $\mathbf{V}$ (e.g., via QR decomposition). Since the eigenspaces associated with spectral clusters are relatively stable and low-dimensional. Therefore, $\mathcal{L}$ provides a direct and precise gradient signal for optimization.

    To handle the complex-valued nature of the Steklov problem within standard deep learning frameworks, we adopt a real-valued tensor representation.

    \textbf{Input Representation:} The complex parameter $n({x}) = n_r({x}) + i n_i({x})$ is stacked with the spatial coordinates. For a mesh with $N$ nodes, the input is formatted as a tensor of shape $[4, N]$, where the four channels correspond to the real part of $n$, the imaginary part of $n$, and the $x, y$ coordinates of the mesh nodes. The coordinates aids the network in capturing geometry-dependent spectral features.

    \textbf{Output and Orthogonalization:} The network outputs a real tensor $\tilde{\mathbf{U}}_{\text{out}}$ of shape $[2k_U, N]$. The first $k_U$ channels represent the real parts of the basis vectors, while the subsequent $k_U$ channels represent the imaginary parts. We reconstruct a complex matrix $\mathbf{U} \in \mathbb{C}^{N \times k_U}$ via:
    \begin{equation}
        \mathbf{U}_{:, j} = (\tilde{\mathbf{U}}_{\text{out}})_{j, :} + i (\tilde{\mathbf{U}}_{\text{out}})_{j+k_U, :}, \quad j=1,\dots,k_U.
    \end{equation}
    To ensure numerical stability and satisfy the input requirements of $\mathcal{L}$, we apply a QR decomposition to $\mathbf{U}$ during the forward pass:
    \begin{equation}
        \mathbf{U} = \mathbf{Q}_U \mathbf{R}_U.
    \end{equation}
    The resulting unitary matrix $\mathbf{Q}_U$ is then passed to the loss function \eqref{eq:L1_loss} for training.

    \subsection{Eigenpair Prediction via Rayleigh-Ritz Method}\label{sec:RayeighRitz}

    Once the predictive network DEN provides an approximate eigenspace $\mathbf{Q}_U \in \mathbb{C}^{N \times k_U}$ (orthonormalized via QR decomposition), we employ the Rayleigh-Ritz procedure to extract the specific eigenpairs. This step can be viewed as a Galerkin projection of the full fine-scale problem onto the low-dimensional subspace spanned by $\mathbf{Q}_U$.

    Given the parameter-dependent mass matrix $\mathbf{M}_n$ (constructed from the input $n({x})$) and the pre-computed constant matrices $\mathbf{S}$ and $\mathbf{M}_{\partial \Omega}$, the full system matrices are $\mathbf{A}(n) = -\mathbf{S} + k^2 \mathbf{M}_n$ and $\mathbf{B} = \mathbf{M}_{\partial \Omega}$. The Rayleigh-Ritz method proceeds as follows.

    \begin{enumerate}
        \item \textbf{Projection (Subspace Assembly):} Project the full matrices onto the predicted subspace to form the reduced stiffness and mass matrices:
        \begin{equation}
            \hat{\mathbf{A}} = \mathbf{Q}_U^H \mathbf{A}(n) \mathbf{Q}_U, \quad 
            \hat{\mathbf{B}} = \mathbf{Q}_U^H \mathbf{B} \mathbf{Q}_U,
        \end{equation}
        where $\hat{\mathbf{A}}, \hat{\mathbf{B}} \in \mathbb{C}^{k_U \times k_U}$ are small, dense matrices.
        
        \item \textbf{Solving the Reduced Problem:} Solve the generalized eigenvalue problem:
        \begin{equation}
            \hat{\mathbf{A}} \mathbf{w}_j = \mu_j \hat{\mathbf{B}} \mathbf{w}_j, \quad j=1, \dots, k_U.
        \end{equation}
        Note that $k_U \ll N$ (e.g., $k_U \sim 10^1$ while $N \sim 10^4$).
        
        \item \textbf{Lifting (Reconstruction):} The Ritz values $\mu_j$'s serve as approximations for the exact eigenvalues $\lambda_j$'s. The corresponding Ritz vectors are lifted back to obtain the approximated eigenfunctions:
        \begin{equation}
            \tilde{u}_j = \mathbf{Q}_U \mathbf{w}_j.
        \end{equation}
    \end{enumerate}
    This procedure ensures that the predicted eigenfunctions $\tilde{u}_j$ strictly reside within the predicted subspace $\mathcal{U}$.



    The online computational cost of proposed method has three components.
    \begin{itemize}
        \item \textbf{Matrix Assembly:} The parameter-dependent matrix $\mathbf{A}(n)$ is constructed via an update of $\mathbf{M}_n$ followed by a sparse addition. This involves only basic sparse matrix arithmetic (e.g., two matrix multiplications and one addition in our implementation).
        \item \textbf{Projection Cost:} The formation of the reduced matrices $\hat{\mathbf{A}}$ and $\hat{\mathbf{B}}$ requires the projection of the sparse full matrices onto the predicted basis $\mathbf{Q}_U$. This operation involves sparse-dense matrix multiplications, with a complexity of approximately $\mathcal{O}(N \cdot k_U^2)$.
        \item \textbf{Solving Cost:} Solving the reduced GEP involves dense matrices of size $k_U \times k_U$, entailing a cost of $\mathcal{O}(k_U^3)$.
    \end{itemize}

    Given that the subspace dimension $k_U$ is a small constant, the total online complexity scales linearly, $\mathcal{O}(N)$. 
    DEN can be viewed as \textbf{``Data-Driven Subspace Prediction"}. It predicts a subspace directly from the input in a single forward pass. The Rayleigh-Ritz step then acts as a final ``polishing" mechanism. 

    \begin{algorithm}
    \caption{Training of DEN}
    \label{alg:DEN_train}
    \begin{algorithmic}
    \STATE{\textbf{Input:} Training samples $\{n_i(x),\, V(n_i(x))\}_{i=1}^{N_{\text{train}}}$,
    hyperparameters (learning rate $\eta$, epochs $N_{\text{ep}}$, rank $K$)}
    \STATE{\textbf{Output:} Trained network $\mathrm{DEN}_{\theta}$}
    \STATE{}
    \STATE{\textit{// Offline: basis construction}}
    \STATE{Assemble snapshot matrix $Y = \bigl[V(n_1(x)),\, \dots,\, V(n_{N_{\text{train}}}(x))\bigr]$}
    \STATE{Compute $\Phi \leftarrow \mathrm{POD}(Y)$ \quad \textit{// first $K$ left singular vectors of $Y$}}
    \STATE{Initialize $\mathrm{DEN}_{\theta}$ with fixed basis $\Phi$}
    \STATE{}
    \STATE{\textit{// Online: network training}}
    \FOR{$\mathrm{ep} = 1, \dots, N_{\mathrm{ep}}$}
        \FOR{$i = 1, \dots, N_{\mathrm{train}}$}
            \STATE{$\hat{U}_i \leftarrow \mathrm{DEN}_{\theta}(n_i(x))$}
            \STATE{$\theta \leftarrow \theta - \eta\,\nabla_{\theta}\,
                   \mathcal{L}\!\left(\hat{U}_i,\, V(n_i(x))\right)$}
        \ENDFOR
    \ENDFOR
    \end{algorithmic}
    \end{algorithm}

    \begin{algorithm}
    \caption{Evaluation of DEN}
    \label{alg:DEN_eval}
    \begin{algorithmic}
    \STATE{\textbf{Input:} $n(x)$, sparse matrices $A(n(x))$ and $B$,
    trained $\mathrm{DEN}_{\theta}$, target number of eigenpairs $k_V$}
    \STATE{\textbf{Output:} Eigenvalues $\{\hat{\lambda}_j\}_{j=1}^{k_V}$ and eigenspace $\hat{V}$}
    \STATE{}
    \STATE{$U \leftarrow \mathrm{DEN}_{\theta}(n(x))$ \quad \textit{// $U \in \mathbb{C}^{N \times k_U}$}}
    \STATE{$Q_U, R_U \leftarrow \mathrm{QR}(U)$ \quad \textit{// economy QR decomposition}}
    \STATE{}
    \STATE{\textit{// Rayleigh--Ritz projection}}
    \STATE{$A_U \leftarrow Q_U^H A(n(x))\, Q_U$, \quad $B_U \leftarrow Q_U^H B\, Q_U$}
    \STATE{Solve $A_U\, \mathbf{u} = \lambda\, B_U\, \mathbf{u}$ for all $k_U$ eigenpairs
           $\{(\hat{\lambda}_j,\, \mathbf{u}_j)\}_{j=1}^{k_U}$}
    \STATE{}
    \STATE{\textit{// Select target eigenpairs}}
    \IF{$k_U > k_V$}
        \STATE{Sort $\{(\hat{\lambda}_j,\, \mathbf{u}_j)\}_{j=1}^{k_U}$ by
               $|\hat{\lambda}_j|$ in ascending order}
        \STATE{$\{\hat{\lambda}_j\}_{j=1}^{k_V} \leftarrow$ first $k_V$ eigenvalues}
        \STATE{$\hat{V} \leftarrow Q_U\,[\mathbf{u}_1,\, \dots,\, \mathbf{u}_{k_V}]$}
    \ELSE
        \STATE{$\{\hat{\lambda}_j\}_{j=1}^{k_V} \leftarrow \{\hat{\lambda}_j\}_{j=1}^{k_U}$}
        \STATE{$\hat{V} \leftarrow Q_U$}
    \ENDIF
    \end{algorithmic}
    \end{algorithm}

    The training and evaluation procedures of DEN are summarized in Algorithms~\ref{alg:DEN_train} and \ref{alg:DEN_eval}, respectively. Algorithm~\ref{alg:DEN_train} iterates over individual training samples. In practice, the training loop processes inputs and outputs in mini-batches, with the gradient update performed once per batch.

    \subsection{Dataset Generation}\label{sec:dataset}

    We generate a diverse dataset of complex-valued refractive-index profiles $n({x})$. The complex parameter is defined as $n({x}) = n_r({x}) + i n_i({x})$. We synthesize the spatial variability of $n({x})$ by superposing random sinusoidal modes. To ensure independence between the real and imaginary components, we generate two distinct random fields, $g_r({x})$ and $g_i({x})$. A random field $g({x})$ is constructed as follows. For integers $m\le j\le M$,
    \begin{equation}\label{eq:nx_gen}
        g({x}) = \sum_{j=m}^{M} \beta^{-j} \Bigl(
            a_j \,\sin\!\bigl(2\pi f_j\,u({x};\alpha_j)+\theta_j\bigr)
            + b_j \,\sin\!\bigl(2\pi f_j\,v({x};\alpha_j)+\varphi_j\bigr)
        \Bigr),
    \end{equation}
    where ${x}=(x_1,x_2)$, and the coordinate rotation is given by:
    \begin{equation}
        u({x};\alpha)=\frac{x_1\cos\alpha+x_2\sin\alpha}{2},\qquad
        v({x};\alpha)=\frac{x_1\sin\alpha+x_2\cos\alpha}{2}.
    \end{equation}
    For each mode $j$, the parameters are sampled as:
    \begin{equation}
        a_j,\;b_j\sim\mathcal{N}(0,1),\qquad
        \theta_j,\;\varphi_j,\;\alpha_j\sim\mathcal{U}(0,2\pi)
    \end{equation}
    with the spatial frequency $f_j=j$. The scalar $\beta\in(0,1)$ controls the spectral decay, favoring low-frequency structures to mimic smooth background media perturbed by fine local variations.

    The final complex parameter $n({x})$ is obtained by independently sampling coefficients for $g_r$ and $g_i$, followed by a min-max normalization:
    \begin{equation}
        n({x}) = \mathcal{T}_{[n_{r}^{\min}, n_{r}^{\max}]}(g_r({x})) + i \cdot \mathcal{T}_{[n_{i}^{\min}, n_{i}^{\max}]}(g_i({x})),
    \end{equation}
    where $\mathcal{T}_{[a,b]}(\cdot)$ denotes the linear scaling operator. In our experiments, we fix both the real range and the imaginary range to $[1.0, 5.0]$.

    Unlike methods that rely on rasterized images (pixels), our problem is defined on a fixed unstructured triangular mesh for $\Omega$ with $N$ nodes, which better capture complex geometries and conform to data generation by FEM. Each realization of $n({x})$ is represented as a complex vector $\mathbf{n} \in \mathbb{C}^N$. Similarly, the target eigenfunctions are stored as vectors $\mathbf{u} \in \mathbb{C}^N$.

    We generate a total of 2000 samples (1600 for training and 400 for testing). To balance computational efficiency with approximation accuracy, the mesh resolution is chosen such that $N \approx 1500$ nodes. This resolution is sufficient to resolve the target eigenfunctions while keeping the training cost of the graph-based neural operators manageable. 
    


    \section{Learnability and Error Analysis}\label{sec:Theory}

    \subsection{Stability of Eigenspace}\label{sec:SubStable}
    We consider the generalized eigenvalue problem
    \begin{equation}
    A x = \lambda B x.
    \end{equation}
    Let $\Gamma$ be a simple closed contour $\Gamma \subset \mathbb{C}$ enclosing a cluster $\Lambda_0$ consisting of $K$ eigenvalues such that $\Gamma \cap \sigma(A,B) = \emptyset$. Then there exists a $\delta > 0$ such that
    \begin{equation}
    \mathrm{dist}\big(\Gamma,\ \sigma(A,B)\setminus \Lambda_0\big) = \delta > 0.
    \end{equation}
    Consequently,
    The resolvent along $\Gamma$ is such that
    \begin{equation}
    R_{\Gamma} := \sup_{z \in \Gamma} \|(Bz - A)^{-1}\| < \infty.
    \end{equation}

    Let $M$ be a perturbation matrix of $A$.
    \begin{assumption}
    \label{assump:smallM}
    The perturbation matrix $M$ satisfies the smallness condition
    \begin{equation}
    \max\{\|M\|, \|M\| \cdot R_\Gamma\} < 1.
    \end{equation}
    \end{assumption}
    \vspace{1em}
    \begin{lemma}[Resolvent Stability]
    \label{lemma:resolvent-stability}
    Under Assumption~\ref{assump:smallM}, the following hold.
    
    \begin{enumerate}
        \item For every perturbation matrix $M$ satisfying Assumption~\ref{assump:smallM},  
        $Bz - A - M$ is invertible for all $z \in \Gamma$, and
        \begin{equation}\label{eq:FiniteResolvent}
        \sup_{z \in \Gamma} \|(B z - A - M)^{-1}\|
        \,\le\, \frac{R_\Gamma}{1 - R_\Gamma \|M\|}.
        \end{equation}
    
        \item The resolvent difference satisfies
        \begin{equation}
        \|(Bz - A)^{-1} - (Bz - A - M)^{-1}\|
        \,\le\, R_\Gamma^2 \|M\| + O(\|M\|^2),
        \qquad z \in \Gamma.
        \end{equation}
    \end{enumerate}
    \end{lemma}

    \begin{proof}
    Let $T(z) = Bz - A$. $T(z)$ is invertible for all $z \in \Gamma$, and $\|T(z)^{-1}\| \le R_{\Gamma}$.
    
    \begin{enumerate}
        \item We write the perturbed matrix as
        \begin{equation}
        Bz - A - M = T(z) - M = T(z) \left( I - T(z)^{-1} M \right).
        \end{equation}
        Let $E = T(z)^{-1} M$. By Assumption~\ref{assump:smallM}, we have
        \begin{equation}
        \|E\| \le \|T(z)^{-1}\| \, \|M\| \le R_\Gamma \|M\| < 1.
        \end{equation}
        Since $\|E\| < 1$, the matrix $I - E$ is invertible, and its inverse can be expressed via the Neumann series $(I - E)^{-1} = \sum_{j=0}^\infty E^j$. Consequently, $Bz - A - M$ is invertible, and its inverse is given by
        \begin{equation}
        (Bz - A - M)^{-1} = (I - E)^{-1} T(z)^{-1}.
        \end{equation}
        Therefore,
        \begin{equation}
        \begin{aligned}
        \|(Bz - A - M)^{-1}\| 
        &\le \|(I - E)^{-1}\| \, \|T(z)^{-1}\| \\
        &\le \frac{1}{1 - \|E\|} R_\Gamma \\
        &\le \frac{R_\Gamma}{1 - R_\Gamma \|M\|}.
        \end{aligned}
        \end{equation}
        
        \item From Part 1 and the Neumann expansion for $(I - E)^{-1}$, we have
        \begin{equation}
        \begin{aligned}
        (Bz - A - M)^{-1} &= \left( I + E + \sum_{j=2}^\infty E^j \right) T(z)^{-1} \\
        &= T(z)^{-1} + E T(z)^{-1} + \left( \sum_{j=2}^\infty E^j \right) T(z)^{-1}.
        \end{aligned}
        \end{equation}
        Recall that $(Bz - A)^{-1} = T(z)^{-1}$. The difference is
        \begin{equation}
        (Bz - A - M)^{-1} - (Bz - A)^{-1} = T(z)^{-1} M T(z)^{-1} + \mathcal{R}(z),
        \end{equation}
        where $\mathcal{R}(z) = \sum_{j=2}^\infty (T(z)^{-1} M)^j T(z)^{-1}$.
        The norm of the first-order term is bounded by:
        \begin{equation}
        \|T(z)^{-1} M T(z)^{-1}\| \le \|T(z)^{-1}\|^2 \|M\| \le R_\Gamma^2 \|M\|.
        \end{equation}
        The remainder term $\mathcal{R}(z)$ is of  $O(\|M\|^2)$ because
        \begin{equation}
        \|\mathcal{R}(z)\| \le \sum_{j=2}^\infty (R_\Gamma \|M\|)^j R_\Gamma = \frac{(R_\Gamma \|M\|)^2 R_\Gamma}{1 - R_\Gamma \|M\|} = O(\|M\|^2),
        \end{equation}
        provided that $R_\Gamma \|M\|$ is sufficiently small.
        Thus 
        \begin{equation}
        \|(Bz - A)^{-1} - (Bz - A - M)^{-1}\| \le R_\Gamma^2 \|M\| + O(\|M\|^2).
        \end{equation}
    \end{enumerate}
    \end{proof}

    \begin{remark}
    The importance of Lemma \ref{lemma:resolvent-stability} in numerical experiments is threefold.
    
    \begin{enumerate}
        \item \textbf{Configuration:} In our study, the reference operator is defined as $A = -S + M_{\bar{n}}$, where $\bar{n}$ represents the statistical mean of the refractive index samples in the dataset. The target spectral set $\Lambda_0 = \{\lambda_1, \dots, \lambda_{12}\}$ comprises the smallest 12 eigenvalues, which constitute the first five eigenvalue clusters.
        
        \item \textbf{Spectral Isolation:} The samples in the dataset satisfy the smallness condition (Assumption \ref{assump:smallM}). Lemma \ref{lemma:resolvent-stability} guarantees that the contour $\Gamma$ remains strictly within the resolvent set for the perturbed operator $A(n)$. This implies that no eigenvalues from outside $\Lambda_0$ (i.e., $\lambda_{13}, \dots$) can cross $\Gamma$ to mix with the target cluster. Thus, the dimensionality of the eigenspace enclosed by $\Gamma$ remains invariant across the entire dataset.
        
        \item \textbf{Uniform Stability:} Equation \eqref{eq:FiniteResolvent} establishes a uniform bound on the resolvent norm along the fixed contour $\Gamma$ for all admissible perturbations. It serves as a necessary condition for Theorem~\ref{thm:projection-stability}, showing the Lipschitz continuity of the mapping from the parameter $n({x})$ to the eigenspace.
    \end{enumerate}
    \end{remark}
    
    Recall the Riesz spectral projection $P_{n}$ associated with the pair $(A + M_{n}, B)$ corresponding to the eigenvalues inside $\Gamma$ is given by
        \begin{equation}
        P_{n} = \frac{1}{2\pi \mathrm{i}} \oint_\Gamma (Bz - A - M_{n})^{-1} B \, \mathrm{d}z.
        \end{equation}
    \begin{theorem}[Stability of Riesz Projections]
    \label{thm:projection-stability}
    Let $M_{n_1}$ be a perturbation matrix satisfying Assumption~\ref{assump:smallM}. Define
    \begin{equation}
    R_{\Gamma,n_1} := \sup_{z \in \Gamma} \|(Bz - A - M_{n_1})^{-1}\|.
    \end{equation}
    Let $M_{n_2}$ be another perturbation matrix and $\Delta M := M_{n_2} - M_{n_1}$.  
    Assume that
    \begin{equation}\label{DMR1}
    \|\Delta M\|\cdot R_{\Gamma,n_1} < 1.
    \end{equation}
    Then the stability estimate holds
    \begin{equation}
    \|P_{n_1} - P_{n_2}\|
    \,\le\,
    \frac{|\Gamma|}{2\pi}\,
    R_{\Gamma,n_1}^2 \|\Delta M\|
    + O(\|\Delta M\|^2).
    \end{equation}
    Furthermore, if $M_{n_1}$ and $M_{n_2}$ arise from material parameters
    $n_1(x), n_2(x) \in C(\Omega)$, then
    \begin{equation}
    \|P_{n_1} - P_{n_2}\|
    \le C \|n_1 - n_2\|.
    \end{equation}
    \end{theorem}

    \begin{proof}
        We proceed in two steps: first, establishing the stability bound with respect to the matrix perturbation $\Delta M$, and second, relating this to the material parameter difference $n_1 - n_2$.

        \paragraph{1. Stability with respect to matrix perturbation}
        Let $T_1(z) = Bz - A - M_{n_1}$ and $T_2(z) = Bz - A - M_{n_2}$. We are interested in estimating the norm of the difference:
        \begin{equation}
        P_{n_1} - P_{n_2} = \frac{1}{2\pi \mathrm{i}} \oint_\Gamma \left( T_1(z)^{-1} - T_2(z)^{-1} \right) B \, \mathrm{d}z.
        \end{equation}
        We view $T_2(z)$ as a perturbation of $T_1(z)$:
        \begin{equation}
        T_2(z) = (Bz - A - M_{n_1}) - (M_{n_2} - M_{n_1}) = T_1(z) - \Delta M.
        \end{equation}
        By Assumption~\ref{assump:smallM} and Lemma~\ref{lemma:resolvent-stability}, $T_1(z)$ is invertible for $z \in \Gamma$ with 
        \begin{equation}
            \sup_{z \in \Gamma} \|T_1(z)^{-1}\| = R_{\Gamma,n_1}.
        \end{equation}
        Factoring out $T_1(z)$, we write
        \begin{equation}
        T_2(z) = T_1(z) \left( I - T_1(z)^{-1} \Delta M \right).
        \end{equation}
        Due to \eqref{DMR1}, the inverse $T_2(z)^{-1}$ can be expanded using the Neumann series:
        \begin{equation}
        \begin{aligned}
        T_2(z)^{-1} 
        &= \left( I - T_1(z)^{-1} \Delta M \right)^{-1} T_1(z)^{-1} \\
        &= \left( I + T_1(z)^{-1} \Delta M + \sum_{j=2}^\infty (T_1(z)^{-1} \Delta M)^j \right) T_1(z)^{-1} \\
        &= T_1(z)^{-1} + T_1(z)^{-1} \Delta M T_1(z)^{-1} + \mathcal{R}(z),
        \end{aligned}
        \end{equation}
        where the remainder term satisfies $\|\mathcal{R}(z)\| = O(\|\Delta M\|^2)$.
        Substituting this expansion into the integral for the projection difference, we obtain
        \begin{equation}
        P_{n_1} - P_{n_2} = - \frac{1}{2\pi \mathrm{i}} \oint_\Gamma \left( T_1(z)^{-1} \Delta M T_1(z)^{-1} + \mathcal{R}(z) \right) B \, \mathrm{d}z.
        \end{equation}
        Taking the norm and using the boundedness of $B$, we have
        \begin{equation}
        \begin{aligned}
        \|P_{n_1} - P_{n_2}\| 
        &\le \frac{|\Gamma|}{2\pi} \sup_{z \in \Gamma} \left\| T_1(z)^{-1} \Delta M T_1(z)^{-1} \right\| \|B\| + O(\|\Delta M\|^2) \\
        &\le \frac{|\Gamma|}{2\pi} \left( \sup_{z \in \Gamma} \|T_1(z)^{-1}\| \right)^2 \|\Delta M\| \|B\| + O(\|\Delta M\|^2) \\
        &\le C_B \frac{|\Gamma|}{2\pi} R_{\Gamma,n_1}^2 \|\Delta M\| + O(\|\Delta M\|^2).
        \end{aligned}
        \end{equation}
        This confirms the first estimate.

        \paragraph{2. Relation to material parameters}
        We now establish the bound in terms of $\|n_1 - n_2\|$. For FEM using linear Lagrange elements, the matrix $M_{n}$ is assembled from local element contributions. Let $\mathcal{T}_h = \{K_e\}_{e=1}^{M_e}$ be the triangulation. 
        
        The local mass matrix $M_{n_i}^{(e)}, i=1,2$ is obtained by
        \begin{equation}
        M_{n_i}^{(e)} = \operatorname{mat}\left( (\mathbf{w}^T \cdot \mathbf{n}_i^{(e)}) \frac{|K^{(e)}|}{3} \right),
        \end{equation}
        where $\mathbf{n}_i^{(e)} \in \mathbb{R}^3$ collects the values of $n(x)$ at the midpoints of the three edges of triangle $e$, $\mathbf{w}^T$ is a fixed weight matrix of size $9 \times 3$, and $\operatorname{mat}(\cdot)$ reshapes the resulting vector into a local $3 \times 3$ element matrix. The global matrix $M_{n_i}$ is given by
        \begin{equation}
        M_{n_i} = \sum_{e=1}^{M_e} \mathcal{A}_e (M_{{n_i}}^{(e)}),
        \end{equation}
        where $\mathcal{A}_e$ is the assembly operator. Since the mapping from the parameter vector $\mathbf{n}$ to the local matrix entries is linear (multiplication by constant weights $\mathbf{y}\mathbf{y}^T$ and area factors), and the global assembly is linear, the mapping $\Phi: \mathbf{n} \mapsto M_n$ is a linear operator between finite-dimensional vector spaces.
        
        Consequently, $\Phi$ is Lipschitz continuous. There exists a constant $C_{\mathrm{mesh}}$ depending on the mesh  (via element areas and weights) such that
        \begin{equation}
        \|\Delta M\| = \|M_{n_1} - M_{n_2}\| = \|\Phi(\mathbf{n}_1) - \Phi(\mathbf{n}_2)\| = \|\Phi(\mathbf{n}_1 - \mathbf{n}_2)\| \le C_{\mathrm{mesh}} \|{n}_1 - {n}_2\|.
        \end{equation}
        Substitution of this into the stability estimate derived in Part 1 yields
        \begin{equation}
        \|P_{n_1} - P_{n_2}\| \le C' R_{\Gamma,n_1}^2 C_{\mathrm{mesh}} \|n_1 - n_2\| + O(\|n_1 - n_2\|^2),
        \end{equation}
        which implies
        \begin{equation}
        \|P_{n_1} - P_{n_2}\| \le C \|n_1 - n_2\|,
        \end{equation}
        for a constant $C > 0$.
    \end{proof}

    The above results justify the learnability of the proposed mapping. First, Lemma \ref{lemma:resolvent-stability} guarantees Uniform Spectral Isolation across the entire dataset. It ensures that the contour $\Gamma$ separates the target cluster from the rest of the spectrum for all admissible samples, preventing any external eigenfunctions from "polluting" the target subspace. This establishes the consistency of the learning target. Second, Theorem \ref{thm:projection-stability} establishes the Lipschitz continuous dependence of the eigenspace on the refractive index $n$. 
    Combining these findings with the fact that the admissible parameters $n({x})$ is defined on a compact set (see Section \ref{sec:dataset}), we conclude that the mapping from the parameter space to the Grassmann manifold is both well-defined and continuous. According to the Universal Approximation Theorems for neural operators, such continuous mappings on compact domains are intrinsically learnable, thereby providing the theoretical guarantee for DEN. 

    \subsection{Eigenvalue Prediction Error Bound}

    Let $A, B \in \mathbb{C}^{n \times n}$ define a generalized eigenvalue problem. Let $\mathcal{V}$ be a $K$-dimensional space (eigenspace) and $\mathcal{U}$ be an approximation space of dimension $K$. Let $V \in \mathbb{C}^{n \times K}$ and $U \in \mathbb{C}^{n \times K}$ be orthonormal bases for $\mathcal{V}$ and $\mathcal{U}$, respectively ($V^HV = U^HU = I_K$). The misalignment between the subspaces is quantified by the sine of the canonical angles:
    \begin{equation}
    \sin\Theta(\mathcal{U}, \mathcal{V}) := \|(I - P_V)P_U\|_2 = \|(I - P_U)P_V\|_2,
    \end{equation}
    where $P_V = VV^H$ and $P_U = UU^H$ are the orthogonal projectors onto $\mathcal{V}$ and $\mathcal{U}$.
    
    Since $A_V = V^HAV$ and $A_U = U^HAU$ are representations in different bases, a direct comparison $A_U - A_V$ is not meaningful without first aligning the two bases. We introduce the \emph{optimal alignment unitary} $\Pi \in \mathbb{C}^{K\times K}$ obtained from the polar decomposition of $V^HU$: if $V^HU = W\Sigma Z^H$ is the SVD, then
    \begin{equation}\label{eq:Pi-def}
    \Pi := WZ^H.
    \end{equation}
    This is the unique unitary matrix that solves the orthogonal Procrustes problem
    \begin{equation}
    \Pi = \arg\min_{Q \in \mathbb{C}^{K\times K},\, Q^HQ = I} \|UQ - V\|_F.
    \end{equation}
    With this alignment, the residual $E := V - U\Pi$ satisfies
    \begin{equation}\label{eq:residual-bound}
    \|V - U\Pi\|_2 \le \sqrt{2}\,\sin\Theta(\mathcal{U},\mathcal{V}),
    \end{equation}
    which follows from the theory of canonical angles (see, e.g., \cite{stewart1990matrix}).
    
    \begin{lemma}[Matrix Projection Error Bound]
    \label{lemma:projection-error}
    Let $A_V = V^HAV$, $A_U = U^HAU$, and let $\Pi$ be the optimal alignment unitary defined in \eqref{eq:Pi-def}. Then
    \begin{align}
    \|A_V - \Pi^H A_U \Pi\|_2 &\le 2\sqrt{2}\,\|A\|_2 \, \sin\Theta(\mathcal{U},\mathcal{V}), \label{eq:A-bound}\\
    \|B_V - \Pi^H B_U \Pi\|_2 &\le 2\sqrt{2}\,\|B\|_2 \, \sin\Theta(\mathcal{U},\mathcal{V}). \label{eq:B-bound}
    \end{align}
    \end{lemma}
    
    \begin{proof}
    Define $\tilde{U} := U\Pi$, so that $\tilde{U}$ is the optimally aligned orthonormal basis for $\mathcal{U}$, and note that
    \begin{equation}
    \Pi^H A_U \Pi = \Pi^H (U^H A U) \Pi = (U\Pi)^H A (U\Pi) = \tilde{U}^H A \tilde{U}.
    \end{equation}
    Hence, with $E = V - \tilde{U}$,
    \begin{equation}
    \begin{aligned}
    A_V - \Pi^H A_U \Pi &= V^HAV - \tilde{U}^HA\tilde{U} \\
    &= V^HA(V - \tilde{U}) + (V^H - \tilde{U}^H)A\tilde{U} \\
    &= V^HAE + E^HA\tilde{U}.
    \end{aligned}
    \end{equation}
    Taking the spectral norm and applying the triangle inequality:
    \begin{equation}
    \|A_V - \Pi^H A_U \Pi\|_2 \le \|V\|_2\,\|A\|_2\,\|E\|_2 + \|E\|_2\,\|A\|_2\,\|\tilde{U}\|_2.
    \end{equation}
    Since $V$ and $\tilde{U} = U\Pi$ are orthonormal ($\|V\|_2 = \|\tilde{U}\|_2 = 1$), we obtain
    \begin{equation}
    \|A_V - \Pi^H A_U \Pi\|_2 \le 2\,\|A\|_2\,\|E\|_2 = 2\,\|A\|_2\,\|V - U\Pi\|_2.
    \end{equation}
    Substituting the bound \eqref{eq:residual-bound} yields
    \begin{equation}
    \|A_V - \Pi^H A_U \Pi\|_2 \le 2\sqrt{2}\,\|A\|_2\,\sin\Theta(\mathcal{U},\mathcal{V}).
    \end{equation}
    The derivation for $B$ is identical.
    \end{proof}

    Let $(M, N)$ be a regular matrix pair of size $K \times K$ that is diagonalizable. That is, there exist nonsingular matrices $X, Y$ such that $Y^HMX = \Lambda$ and $Y^HNX = \Omega$ are diagonal. Let $\{\lambda_i = \Lambda_{ii}/\Omega_{ii}\}$ be the generalized eigenvalues. Define the eigenvalue condition number for the $i$-th mode as:
    \begin{equation}
    \nu_i := \frac{\|x_i\|_2 \|y_i\|_2}{\sqrt{|\Lambda_{ii}|^2 + |\Omega_{ii}|^2}},
    \end{equation}
    where $x_i, y_i$ are the $i$-th columns of $X$ and $Y$.
    
    \begin{lemma}[Spectral Variation for Diagonalizable Pencils]
    \label{lemma:perturbation}
    Consider a perturbed pair $(\tilde{M}, \tilde{N}) = (M+E, N+F)$. The spectral variation, defined as the maximum distance from a true eigenvalue to the set of perturbed eigenvalues $\{\tilde{\lambda}_j\}$, satisfies
    \begin{equation}
    \mathrm{sv}_{(M,N)}[(\tilde{M}, \tilde{N})] 
    := \max_i \min_j |\tilde{\lambda}_j - \lambda_i| 
    \le \left(\max_i \nu_i\right) \sqrt{\|E\|_2^2 + \|F\|_2^2}.
    \end{equation}
    \end{lemma}
    
    \begin{proof}
    Corollary 2.8 of \cite{stewart1990matrix} gives 
    \begin{equation}
        \mathrm{sv}_{(M,N)}[(\tilde{M}, \tilde{N})] \leq 
        \left(\max_i \nu_i\right) \|[E\ F]\|_2.
    \end{equation}
     The bound follows by
     \begin{equation}
         \|[E\ F]\|_2^2=\|EE^H+FF^H\|_2 \leq \|EE^H\|_2+\|FF^H\|_2 = \|E\|_2^2 + \|F\|_2^2.
     \end{equation}
    \end{proof}

    \begin{lemma}[Unitary Invariance of Spectral Variation]
    \label{lemma:unitary-invariance}
    For any unitary matrix $\Pi \in \mathbb{C}^{K\times K}$ and any matrix pairs $(M,N)$ and $(\tilde{M},\tilde{N})$ of size $K\times K$,
    \begin{equation}
    \mathrm{sv}_{(M,N)}[(\tilde{M}, \tilde{N})] = \mathrm{sv}_{(M,N)}[(\Pi^H\tilde{M}\Pi,\, \Pi^H\tilde{N}\Pi)].
    \end{equation}
    \end{lemma}
    
    \begin{proof}
    The generalized eigenvalues of the pair $(\Pi^H\tilde{M}\Pi,\, \Pi^H\tilde{N}\Pi)$ are the same as those of $(\tilde{M},\tilde{N})$, since if $\tilde{M}v = \tilde{\lambda}\,\tilde{N}v$, then $(\Pi^H\tilde{M}\Pi)(\Pi^Hv) = \tilde{\lambda}\,(\Pi^H\tilde{N}\Pi)(\Pi^Hv)$. Because the spectral variation depends only on the eigenvalues of the two pairs, the identity follows.
    \end{proof}

    \begin{theorem}[Linear Convergence of Ritz Values]
    \label{thm:main-result}
    Consider the generalized eigenvalue problem $Ax = \lambda Bx$. Suppose the target eigenspace $\mathcal{V}$ corresponds to a cluster of $K$ eigenvalues that is separated from the rest of the spectrum. Let $(A_V, B_V)$ be the exact restriction and $(A_U, B_U)$ be the Rayleigh--Ritz approximation on $\mathcal{U}$. Assume $(A_V, B_V)$ is diagonalizable with maximal condition number $\nu_{\max} = \max_i \nu_i$. Then, the error in the Ritz values $\{\mu_j\}_{j=1}^K$ relative to the exact eigenvalues $\{\lambda_i\}_{i=1}^K$ is controlled linearly by the subspace misalignment:
    \begin{equation}
    \mathrm{sv}_{(A_V,B_V)}[(A_U, B_U)] \le 2\sqrt{2}\,\nu_{\max} \cdot  \sin\Theta(\mathcal{U},\mathcal{V}) \cdot \sqrt{\|A\|_2^2 + \|B\|_2^2}.
    \end{equation}
    \end{theorem}
    
    \begin{proof}
    Let $\Pi$ be the optimal alignment unitary from \eqref{eq:Pi-def}. By Lemma~\ref{lemma:unitary-invariance}, the spectral variation is invariant under unitary congruence of the approximating pair:
    \begin{equation}
    \mathrm{sv}_{(A_V,B_V)}[(A_U, B_U)] = \mathrm{sv}_{(A_V,B_V)}[(\Pi^H A_U \Pi,\, \Pi^H B_U \Pi)].
    \end{equation}
    We now interpret the aligned pair $(\Pi^H A_U \Pi,\, \Pi^H B_U \Pi)$ as a perturbation of the exact pair $(A_V, B_V)$. Define the perturbation matrices
    \begin{equation}
    E = \Pi^H A_U \Pi - A_V, \quad F = \Pi^H B_U \Pi - B_V.
    \end{equation}
    By Lemma~\ref{lemma:perturbation}, the spectral variation is bounded by
    \begin{equation}
    \mathrm{sv}_{(A_V,B_V)}[(\Pi^H A_U \Pi,\, \Pi^H B_U \Pi)] \le \nu_{\max}\,\sqrt{\|E\|_2^2 + \|F\|_2^2}.
    \end{equation}
    Applying Lemma~\ref{lemma:projection-error}, we have
    \begin{equation}
    \|E\|_2 = \|\Pi^H A_U \Pi - A_V\|_2 \le 2\sqrt{2}\,\|A\|_2\,\sin\Theta, \quad
    \|F\|_2 \le 2\sqrt{2}\,\|B\|_2\,\sin\Theta.
    \end{equation}
    Substituting into the spectral variation bound:
    \begin{equation}
    \begin{aligned}
    \mathrm{sv}_{(A_V,B_V)}[(A_U, B_U)] 
    &\le \nu_{\max}\,\sqrt{(2\sqrt{2}\,\|A\|_2\,\sin\Theta)^2 + (2\sqrt{2}\,\|B\|_2\,\sin\Theta)^2} \\
    &= 2\sqrt{2}\,\nu_{\max}\,\sin\Theta\,\sqrt{\|A\|_2^2 + \|B\|_2^2}.
    \end{aligned}
    \end{equation}
    Thus, the error in the generalized Ritz values is $O(\sin\Theta)$, i.e., of the same order as the subspace approximation error.
    \end{proof}

    \section{Numerical Experiments}\label{sec:Exp}

    In this section, we present a comprehensive evaluation of DEN. We begin by detailing the experimental setup, training hyperparameters, and the evaluation metrics. Subsequently, we report the numerical results, demonstrating that DEN achieves high accuracy in predicting both eigenspaces and eigenpairs. 
    We extend the evaluation to investigate the robustness of DEN under varying magnitudes of the wavenumber $k$, demonstrating that DEN has exceptional performance in low-frequency regimes and maintains competitive accuracy for larger wavenumbers through the adoption of a subspace embedding strategy. Finally, we conducted an ablation study.

    \subsection{Settings and Evaluation Metrics}

    We implemented DEN using the PyTorch. The specific hyperparameters for the network architecture and the training procedure are summarized in Table \ref{tab:settings}. The experiments were conducted with a single NVIDIA GeForce RTX 3090 GPU (24GB VRAM). The training process for the entire dataset converges rapidly, requiring approximately 3 to 4 minutes for 201 epochs. On a test set of 400 samples, the total inference time is 3.9851s, of which 0.1468s is spent on DEN forward passes and 3.8383s on the Rayleigh--Ritz procedure, yielding an overall throughput of 100.3739 samples/s.
    
    \begin{table}[]
        \centering
        \begin{tabular}{|c|cc|}
        \hline
                                            & Parameter                             & Value     \\ \hline
        \multirow{6}{*}{Network settings}   & Number of spectral convolution layers & 4         \\
                                            & Number of hidden channels             & 32        \\
                                            & Dimension of POD basis                & 240       \\
                                            & Rank of mode-mixing matrix            & 32        \\
                                            & Bandwith of mode-mixng matrix         & 5         \\
                                            & Activation function                   & GeLU      \\ \hline
        \multirow{6}{*}{Training settings} & Batch size                            & 32        \\
                                            & Initial learning rate                 & 0.01      \\
                                            & Number of epochs                      & 201       \\
                                            & Learning rate decay step and rate     & (20, 0.8) \\
                                            & Optimizer                             & Adam      \\
                                            & Weight decay                          & 1e-6      \\ \hline
        \end{tabular}
        \caption{Network and training settings.}
    \end{table}\label{tab:settings}

    To provide a comprehensive assessment of DEN, we evaluate the spectral accuracy (eigenvalues), the geometric accuracy (eigenspaces), and the basis representation accuracy (eigenfunctions).

    \paragraph{1. Eigenvalue Metrics}
    For the $k$-th eigenvalue across the test dataset (of size $N_{test}$), we employ the following metric:
    \begin{itemize}
        \item \textbf{Mean Absolute Error (MAE):}
        \begin{equation}
            \text{MAE}_k = \frac{1}{N_{test}} \sum_{i=1}^{N_{test}} |\lambda_k^{(i)} - \hat{\lambda}_k^{(i)}|,
        \end{equation}
        where $\lambda_k^{(i)}$ and $\hat{\lambda}_k^{(i)}$ are the ground truth and predicted eigenvalues for the $i$-th sample, respectively. 
        
        
    \end{itemize}

    \paragraph{2. Eigenfunction Metrics}
    Directly comparing a predicted eigenfunction $\tilde{u}_k$ with the ground truth $u_k$ is problematic for non-self-adjoint problems due to mode switching and index shifting near eigenvalue crossings. The indices of eigenvalues sorted by modulus may swap due to small perturbations, making index-to-index comparison unreliable.

    Instead, we adopt a projection-based evaluation strategy. We assess how well the \textit{true} eigenfunction $u_k$ can be represented by the predicted subspace $\mathcal{U}$. If $u_k$ lies almost entirely within $\mathcal{U}$, then the Rayleigh-Ritz procedure will accurately recover it, regardless of its index.
    Let $u_k$ be the ground truth eigenfunction and $\mathcal{P}_{\mathcal{U}} u_k = \mathbf{Q}_U \mathbf{Q}_U^H u_k$ be its projection onto $\mathcal{U}$. We use:

    \begin{itemize}
        
        
        \item \textbf{Relative $L^1$ Error (RelL1):}
        \begin{equation}
            \text{RelL1}_k = \frac{\|u_k - \mathcal{P}_{\mathcal{U}} u_k\|_1}{\|u_k\|_1}.
        \end{equation}
    \end{itemize}

    \paragraph{3. Eigenspace Metrics}
    \begin{itemize}
        \item \textbf{Subspace Loss ($\mathcal{L}$):} As defined in \eqref{eq:L1_loss}, this metric serves as the training objective and quantifies the alignment of the orthonormal bases.
        
        \item \textbf{Projection Distance ($d_{pr}$):} 
        \begin{equation}
            d_{pr}(\mathcal{U}, \mathcal{V}) = \|\mathbf{P}_U - \mathbf{P}_V\|_2 = \sin(\theta_{\max}),
        \end{equation}
        where $\theta_{\max}$ is the largest principal angle between the subspaces. 
        
        \item \textbf{Chordal Distance ($d_{ch}$):} 
        \begin{equation}
            d_{ch}(\mathcal{U}, \mathcal{V}) = \frac{1}{\sqrt{2}} \|\mathbf{P}_U - \mathbf{P}_V\|_F = \sqrt{\sum_{j} \sin^2(\theta_j)}.
        \end{equation}
    \end{itemize}


    \subsection{Numerical Results}

    We present the results for $\Omega = [0, 1]^2$. The target output consists of the eigenspace spanned by the first 12 eigenfunctions. Note that other domains including a disk, a triangular domain and a L-shaped domain were tested and the results are similar.
   

    
    Table~\ref{tab:err_ev_ef} summarizes the prediction errors on the test set. The test loss $\mathcal{L}$ converges to a magnitude of $10^{-5}$, indicating effective generalization to unseen refractive indices. Geometrically, the projection distance $d_{pr} \approx 0.0020$ corresponds to a maximum principal angle of merely $0.11^\circ$, rendering the predicted subspace $\mathcal{U}$ nearly indistinguishable from the true invariant subspace $\mathcal{V}$ in orientation. The low chordal distance $d_{ch} \approx 0.0040$ further indicates negligible projection errors for the subsequent Rayleigh-Ritz procedure. The results confirms that, while individual eigenvectors may be numerically unstable due to non-self-adjointness, the underlying eigenspace is stable and learnable with high precision.

    \begin{table}[]
        \centering
        \begin{tabular}{|c|ccccc|}
            \hline
            \textbf{Index} & MAE ($\times 10^{-3}$) & RelL1 $(\times 10^{-3})$ & $\mathcal{L}$                                  & $d_{pr}$                               & $d_{ch}$                               \\ \hline
            1              & 0.9044                 & 1.1648                  & \multirow{12}{*}{1.65$\times 10^{-5}$} & \multirow{12}{*}{1.97$\times 10^{-3}$} & \multirow{12}{*}{4.03$\times 10^{-3}$} \\
            2              & 0.8942                 & 1.0954                  &                                        &                                        &                                        \\
            3              & 0.2758                 & 1.0797                  &                                        &                                        &                                        \\
            4              & 7.3634                 & 1.1048                  &                                        &                                        &                                        \\
            5              & 0.6809                 & 1.2536                  &                                        &                                        &                                        \\
            6              & 0.6883                 & 1.2546                  &                                        &                                        &                                        \\
            7              & 0.3966                 & 1.2630                  &                                        &                                        &                                        \\
            8              & 0.4256                 & 1.2707                  &                                        &                                        &                                        \\
            9              & 0.2067                 & 1.7668                  &                                        &                                        &                                        \\
            10             & 0.1820                 & 1.8457                  &                                        &                                        &                                        \\
            11             & 0.1750                 & 1.5992                  &                                        &                                        &                                        \\
            12             & 0.1952                 & 1.9085                  &                                        &                                        &                                        \\ \hline
            \end{tabular}
    \caption{Prediction error of eigenvalues, eigenfunctions and eigenspace.}
    \end{table}\label{tab:err_ev_ef}

    The predicted eigenvalues and projected eigenfunctions also exhibit high accuracy. The eigenvalue MAE remains consistently small ($\sim 10^{-4}$) with no significant degradation for higher modes, except for a mild outlier at Index 4 ($\text{MAE} = 7.36 \times 10^{-3}$) attributable to the intrinsically large spread of $\lambda_4$ across different refractive index profiles, as evidenced by the scattered distribution of the corresponding cluster in Fig.~\ref{fig:ev_cluster}. For the dense clusters $\{5,\ldots,8\}$ and $\{9,\ldots,12\}$, the Rayleigh-Ritz procedure effectively resolves the intra-cluster fine structure, benefiting from the accuracy of the predicted invariant subspace. The eigenfunction RelL1 errors remain uniformly below $0.002$ across all modes, confirming that the global structural integrity of the projected eigenfunctions is well preserved throughout the spectrum.

    Figure \ref{fig:ef_sample} provides a qualitative assessment of the subspace prediction quality. We visualize the eigenfunctions corresponding to indices 1, 3, 4, 5, and 9, which serve as representatives for the five distinct eigenvalue clusters identified in the spectrum.

    \textbf{Visual Inspection:}
    As illustrated in the top two rows, the projected eigenfunctions ($\mathcal{P}_{\mathcal{U}} u_k$) are visually indistinguishable from the ground truth ($u_k$) for both real and imaginary components. This confirms that the predicted subspace $\mathcal{U}$ successfully captures the dominant geometric features and oscillating patterns of the true eigenfunctions.

    \textbf{Error Distribution:}
    The pointwise relative error maps (bottom row) quantify this similarity, revealing that the discrepancy is strictly controlled, with maximum error magnitudes remaining on the order of $10^{-3}$. 

    \begin{figure}[!htbp]
        \centering
        \begin{minipage}[t]{1\textwidth}
            \centering
            \includegraphics[width=1\textwidth]{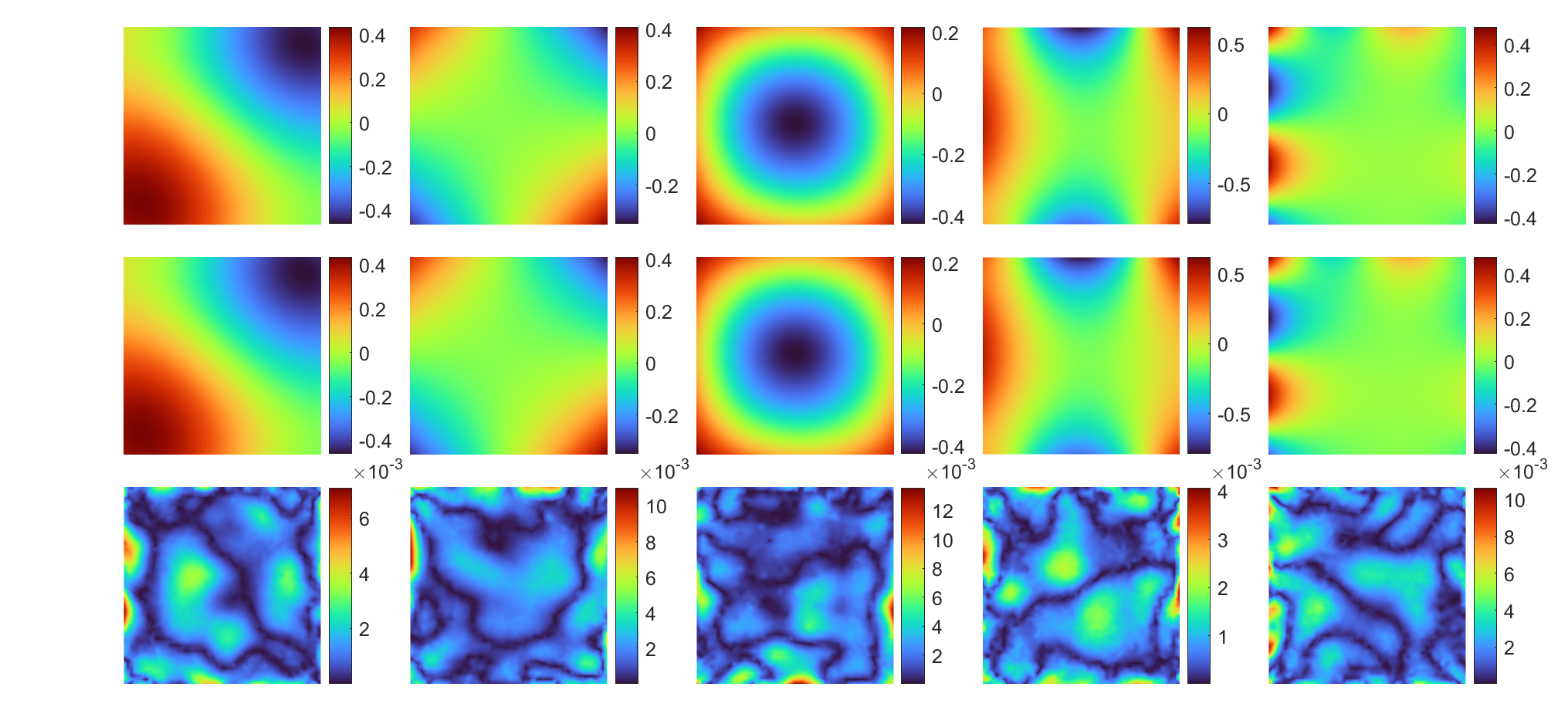}
            \subcaption{Real part of eigenfunctions (top), their projection onto the predicted eigenspace (middle), and relative errors (bottom).}
        \end{minipage} \\
        \begin{minipage}[t]{1\textwidth}
            \centering
            \includegraphics[width=1\textwidth]{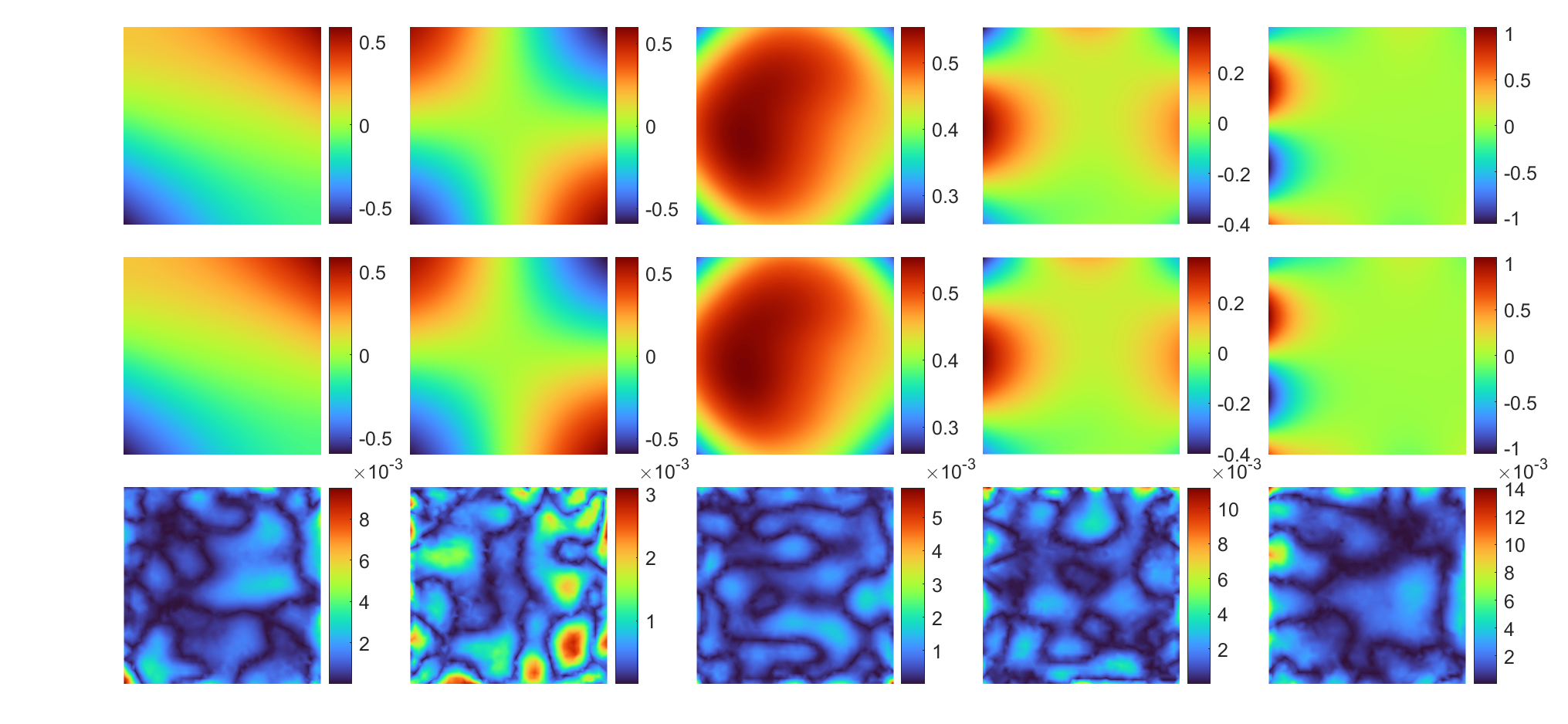}
            \subcaption{Imaginary part of eigenfunctions (top), their projection onto the predicted eigenspace (middle), and relative errors (bottom).}
        \end{minipage} \\
    \caption{Visualization of predicted eigenspace.}
    \label{fig:ef_sample}
    \end{figure}

    \subsection{Robustness Analysis} \label{sec:exp_k_variation}

    In the preceding experiments, the wavenumber was fixed at $k=1$. However, the spectral properties of the Steklov problem depends on $k$. When $k$ is small, the problem can be viewed as a perturbation of the Laplace-Steklov problem, where symmetries in the domain (e.g., the square domain) lead to degenerate eigenvalues that split into tight clusters under the perturbation of $k^2n({x})$. As $k$ increases, these clusters tend to dissolve as the spectrum spreads. More critically, the increase of wavenumber leads to a densification of the spectrum, causing the uniform spectral isolation assumption (Assumption \ref{assump:smallM}) to fail. In this larger frequency regime, it might not exist a contour that encloses exactly the first $K$ eigenvalues for all samples without modes crossing the boundary. Consequently, the mapping from parameters to the exact $K$-dimensional eigenspace is discontinuous due to frequent mode switching.

        To address this challenge, we leverage the concept of subspace embedding \cite{fanaskov2025deep}. Instead of attempting to predict the exact, potentially discontinuous $K$-dimensional target space, we train the network to predict a subspace $\mathcal{U}_{K_{out}}$ such that $K_{out}:=\dim(\mathcal{U}_{K_{out}}) > K$. The rationale is that this larger "ansatz space" can encompass the union of all dominant eigenfunctions that might fluctuate into the range of interest across the parameter distribution. The subsequent Rayleigh-Ritz procedure then serves as a robust filter to extract the desired $K$ eigenpairs.

        We evaluate the model's robustness for $k^2 \in \{0.25, 1, 4, 10\}$. For the lower frequencies ($k^2=0.25, 1$) we set the output dimension equal to the target dimension ($K_{out}=K=12$). For the larger frequencies ($k^2=4, 10$), when mode mixing is expected, we employ the embedding strategy with increased output dimensions of $K_{out}=24$ and $K_{out}=48$, respectively.

        Figure \ref{fig:wave2} illustrates the performance trends of representative metrics (Eigenvalue MAE, Eigenfunction RelL1, and Subspace Chordal Distance) as the wavenumber increases. We observe a general trend of performance degradation across all metrics as $k^2$ grows. This is anticipated, as larger wavenumbers correspond to more oscillatory behavior, which naturally pose greater challenges for both the neural operator and the finite element ground truth generation. However, despite the relative degradation, the absolute errors remain satisfactory even for the challenging $k^2=10$ case, with the eigenvalue MAE maintained around $10^{-2}$.

        To further validate the necessity of the embedding strategy, Figure \ref{fig:k10_dim} details the impact of the output subspace dimension $K_{out}$ on model performance for $k^2=10$. The results show a clear monotonic improvement in accuracy as $K_{out}$ increases from 12 to 48. This empirical evidence confirms that a larger predicted subspace effectively accommodates the "mode switching" phenomenon, capturing necessary spectral information that would otherwise be lost if the network were constrained to a fixed $K$-dimensional output. We note that experiments for significantly larger wavenumbers ($k^2 \gg 10$) are currently omitted, as the standard FEM ground truth generation begins to suffer from non-negligible numerical dispersion errors.

        \begin{figure}[htbp]
            \centering
            \includegraphics[width=1\textwidth]{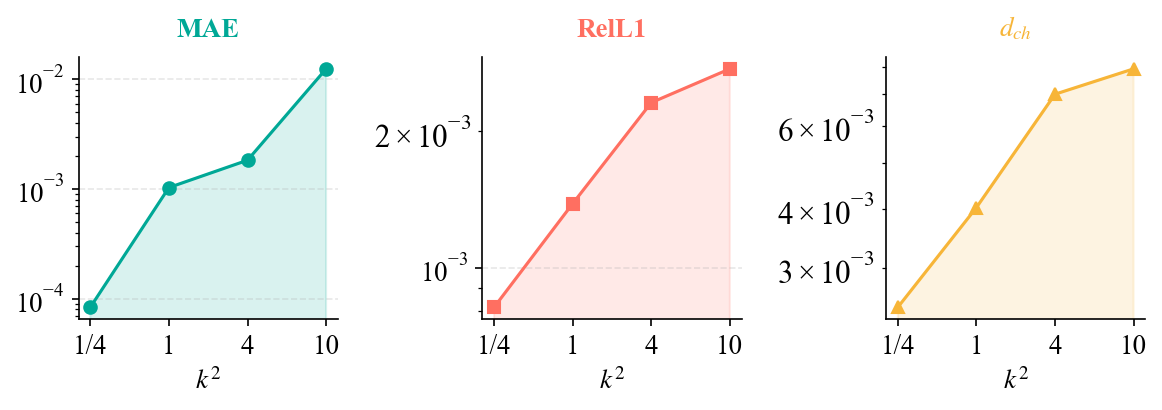}
            \caption{Performance trends of DEN across varying wavenumbers $k^2 \in \{0.25, 1, 4, 10\}$. Metrics shown are Eigenvalue MAE (averaged over top 12), Eigenfunction RelL1 (averaged), and Subspace Chordal Distance $d_{ch}$. Subspace embedding ($K_{out}>12$) is employed for $k^2 \ge 4$.}
            \label{fig:wave2}
        \end{figure}

        \begin{figure}[htbp]
            \centering
            \includegraphics[width=1\textwidth]{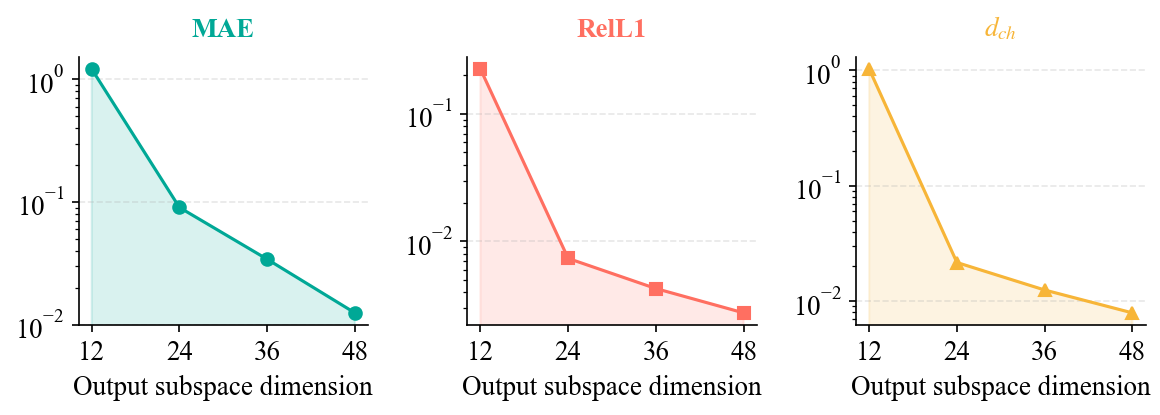}
            \caption{Impact of output subspace dimension $K_{out}$ on model performance for the high-frequency case $k^2=10$. Increasing the embedding dimension effectively mitigates errors caused by spectral crossing.}
            \label{fig:k10_dim}
        \end{figure}

    \subsection{Ablation Study}\label{sec:ablation}
    To rigorously validate the effectiveness of the proposed architectural components, specifically the geometry-adaptive POD basis and the low-rank banded cross-mode mixing, we conduct a series of ablation studies. We compare the full \textbf{DEN} model against several variants and baselines. The results are summarized in Table \ref{tab:ablation}.
    \begin{table}[htbp]
        \centering
        \caption{Ablation study and baseline comparison. \textbf{DEN} denotes the proposed model using $\text{POD}(\mathbf{Y})$ basis and Low-Rank Banded mixing. 
        \textbf{DEN-BFR}: Banded mixing without low-rank constraint. 
        \textbf{DEN-Dense}: Global mixing without banded or low-rank constraints. 
        \textbf{DEN-Diag}: No cross-mode mixing (diagonal only). 
        \textbf{DEN-XY}: Uses joint POD basis of inputs and outputs(equivalent to PODNO with mixing). 
        \textbf{DEN-Lap}: Uses Graph Laplacian eigenfunctions as basis (equivalent to G-FuNK with mixing). 
        \textbf{PODNO/G-FuNK}: Standard baselines without mixing mechanisms. 
        \textbf{FNO}: Standard Fourier Neural Operator with interpolation.}
        \label{tab:ablation}
        \renewcommand{\arraystretch}{1.2}
        \setlength{\tabcolsep}{8pt}
        \begin{tabular}{lcccc}
        \hline
        \textbf{Model} & \textbf{MAE} ($\times 10^{-3}$) & \textbf{RelL1} ($\times 10^{-3}$) & ${d_{ch}}$ ($\times 10^{-3}$) & \textbf{Params} \\ 
        \hline
        \textbf{DEN} & \textbf{1.0324} & \textbf{1.3839} & \textbf{4.0260} & 2,200,536 \\
        DEN-BFR & 1.8937 & 1.5522 & 4.5110 & 2,488,536 \\
        DEN-Dense & 1.0514 & 1.8356 & 5.5449 & 2,488,536 \\
        DEN-Diag & 2.1927 & 2.2502 & 6.5927 & 2,027,736 \\
        \hline
        DEN-XY & 1.0334 & 1.5946 & 4.6279 & 2,200,536 \\
        DEN-Lap & 2.2530 & 2.3331 & 6.8763 & 2,200,536 \\
        \hline
        PODNO & 2.7347 & 2.7719 & 8.1112 & 2,200,536 \\
        G-FuNK & 3.7542 & 3.1976 & 9.3496 & 2,200,536 \\
        FNO & 7.4960 & 4.7606 & 14.0322 & 7,135,305 \\ 
        \hline
        \end{tabular}
    \end{table}

    \paragraph{Impact of Basis Selection}
    We first evaluate the influence of the spectral basis $\Psi$ on the accuracy of the prediction. The results demonstrate that the proposed target-derived basis, $\text{POD}(\mathbf{Y})$, significantly outperforms the geometry-driven graph Laplacian basis (DEN-Lap), reducing the MAE from $2.25$ to $1.03$. While the Laplacian basis generalizes well to the domain geometry, the POD basis optimally encodes the specific energy spectrum of the Steklov eigenfunctions, confirming the theoretical optimality discussed in Section \ref{sec:basis}. We also experimented with a basis derived from the joint input-output snapshot matrix (DEN-XY), which yielded slightly inferior performance compared to using $\text{POD}(\mathbf{Y})$ alone. This indicates that incorporating input features (refractive index) into the basis construction likely introduces unnecessary noise to the output space representation, as the primary objective is to span the solution manifold rather than the parameter manifold. Furthermore, the standard FNO, which necessitates interpolation onto a uniform grid, exhibits the highest error ($7.49$). This sharp contrast underscores the superiority of our mesh-agnostic spectral approach over methods constrained by FFT requirements and interpolation artifacts.

    \paragraph{Importance of Cross-Mode Mixing}
    Next, we investigate the necessity of the proposed cross-mode mixing mechanism. The variant without mixing (DEN-Diag) exhibits a substantial increase in error compared to the full DEN (MAE: $2.19$ vs. $1.03$). Similarly, standard spectral baselines lacking explicit mixing, such as PODNO and G-FuNK, perform significantly worse. These results confirm that the diagonal assumption inherent in standard spectral layers is insufficient for non-self-adjoint operators, for which strong spectral coupling exists. 

    Regarding the efficiency of the mixing strategy, we compared our Low-Rank Banded approach against two denser alternatives: full-rank banded mixing (DEN-BFR) and unconstrained global mixing (DEN-Dense). Interestingly, the proposed DEN outperforms both heavier variants. The unconstrained DEN-Dense likely suffers from optimization difficulties or overfitting due to the excessive parameterization of spectral interactions. By enforcing a banded low-rank structure, DEN not only reduces the parameter count (saving $\sim 0.3$ million parameters compared to the full rank settings) but also imposes a beneficial inductive bias that focuses on physically relevant, local spectral interactions. In summary, the proposed DEN achieves the best trade-off between accuracy and model complexity.

    \section{Conclusions and Discusions}\label{sec:Concl}

    In this work, we considered the spectral inference for parametric non-self-adjoint eigenvalue problems. We proposed DEN, an architecture focusing on mesh adaptability, optimal basis selection in the frequency domain, and explicit cross-mode interactions. To overcome the discontinuity of individual eigenfunctions caused by spectral crossings and non-self-adjointness, we propose subspace learning other than direct eigenfunction regression. By predicting the eigenspace associated to the target eigenvalues and subsequently employing the Rayleigh-Ritz procedure, we robustly recover the eigenvalues and eigenfunctions with high fidelity.

    For a specified spectral range (e.g., the first $k=12$ eigenvalues), we prove that a contour exists such that the resolvent remains bounded for admissible parameters. In addition, we show that the Lipschitz continuous dependence of the eigenspace on the parameters and derive explicit upper bounds for the eigenvalue prediction error.
    Numerical experiments demonstrate that DEN achieves exceptional accuracy in eigenspace prediction. 
    These results characterize DEN as a fast, accurate, and flexible surrogate solver. A study of the impact of the wavenumber $k$ on the model's performance is carried out, indicating that the proposed method yields highly accurate predictions for small wavenumbers. For larger wavenumbers with increasing spectral instability, satisfactory prediction fidelity is preserved by employing a subspace embedding strategy to accommodate the expanded spectral content.

    The proposed framework holds significant promise for downstream applications, including the construction of enrichment basis functions for Generalized Finite Element Methods (GFEM), Model Order Reduction (MOR) for large-scale PDE systems, and fast forward solvers for inverse spectral problems.

\bibliographystyle{siamplain}
\bibliography{refs}

\end{document}

%% file: ex_shared_v1.tex
\usepackage{lipsum}
\usepackage{amsfonts}
\usepackage{graphicx}
\usepackage{epstopdf}
\usepackage{algorithmic}
\ifpdf
  \DeclareGraphicsExtensions{.eps,.pdf,.png,.jpg}
\else
  \DeclareGraphicsExtensions{.eps}
\fi

\newsiamremark{remark}{Remark}
\newsiamremark{hypothesis}{Hypothesis}
\crefname{hypothesis}{Hypothesis}{Hypotheses}
\newsiamthm{claim}{Claim}
\newsiamremark{fact}{Fact}
\crefname{fact}{Fact}{Facts}

\headers{DEN for Parametric Non-self-adjoint Eigenvalue Problems}{H. Li, J. Sun, and Z. Zhang}

\title{Deep Eigenspace Network for Parametric Non-self-adjoint Eigenvalue Problems\thanks{Submitted to the editors \textcolor{black}{DATE}.
\funding{\textcolor{black}{JS was partially supported by NSF Grant DMS-2109949 and SIMONS Foundation Collaboration Grant 711922.  ZZ was partially supported by the National Natural Science Foundation of China  (Project 92470103), the Hong Kong RGC grant (Projects 17304324 and 17300325), the Seed Funding Programme for Basic Research (HKU), and the Hong Kong RGC Research Fellow Scheme 2025.  }}}}

\author{Haoqian Li\thanks{Department of Mathematics, The University of Hong Kong, Pokfulam Road, Hong Kong SAR, China. 
  \email{lihaoqianlhq@connect.hku.hk}.}
\and Jiguang Sun\thanks{Corresponding author. Department of Mathematical Sciences, Michigan Technological University, Houghton, MI 49931, U.S.A. 
  \email{jiguangs@mtu.edu}.}
\and Zhiwen Zhang\thanks{Corresponding author. Department of Mathematics, The University of Hong Kong, Pokfulam Road, Hong Kong SAR, China. AND Materials Innovation Institute for Life Sciences and Energy (MILES), HKU-SIRI, Shenzhen, P.\,R. China.
  \email{zhangzw@hku.hk}}}

\usepackage{amsopn}
